\documentclass [a4paper,12pt]{amsart}
\usepackage{amsthm}
\usepackage{amsmath}
\usepackage{amssymb}
\usepackage[ansinew]{inputenc}
\usepackage{amscd}
\usepackage[ansinew]{inputenc}
\usepackage[mathscr]{euscript}
\usepackage{color, lipsum}
\usepackage[backref]{hyperref}
\makeatletter
\def\subsection{\@startsection{subsection}{2}%
  \z@{.5\linespacing\@plus.7\linespacing}{.3\linespacing}%
  {\normalfont\bfseries}}
\makeatother

\newtheorem{thm}{Theorem}[section]
\newtheorem{cor}[thm]{Corollary}
\newtheorem{lem}[thm]{Lemma}
\newtheorem{prop}[thm]{Proposition}

\theoremstyle{definition}
\newtheorem{ex}[thm]{Example}

\theoremstyle{definition}
\newtheorem{defn}[thm]{Definition}

\theoremstyle{definition}
\newtheorem{rem}[thm]{Remark}

\theoremstyle{definition}

\def\Q{\mathbb Q}
\def\C{\mathbb C}
\def\R{\mathbb R}
\def\K{\mathbb K}
\def\Z{\mathbb Z}

\def\dim{\operatorname{dim}}

\def\supp{\operatorname{supp}}

\def\ord {\mathrm{ord}}

\def\O{\mathcal O}

\def\imag{\mathbf i}
\def\p{\operatorname{p}}
\def\v{\mathbf v}
\def\g{{g}}

\def\B{\mathcal B}
\def\G{\mathbf G}

\def\m{\mathbf m}
\def\F{\mathscr F}

\def\V{\mathbf V}

\def\sL{\textnormal{\texttt L}}

\def\LL{\mathcal L}

\def\k{\mathbf k}

\def\geq{\geqslant}
\def\leq{\leqslant}

\def\*{{\color{red}\blacksquare}}
\def\+{\color{green}\blacksquare}

\subjclass[$2000$ Mathematics Subject Classification]{Primary
13H15; Secondary 13B22, 32S05}

\begin{document}



\title[The sequence of mixed \L ojasiewicz exponents]{The sequence of mixed \L ojasiewicz exponents
\\ \vskip4pt associated to pairs of ideals}



\author{Carles Bivi\`a-Ausina}
\address{
Institut Universitari de Matem\`atica Pura i Aplicada,
Universitat Polit\`ecnica de Val\`encia,
Cam\'i de Vera, s/n,
46022 Val\`encia,
Spain}
\email{carbivia@mat.upv.es}

\keywords{\L ojasiewicz exponents, integral closure of ideals, mixed multiplicities of ideals, monomial ideals, Newton polyhedra}

\thanks{The author was partially supported by DGICYT Grant MTM2015-64013-P}

\begin{abstract}
We analyze the sequence $\LL^*_J(I)$ of mixed \L ojasiewicz exponents
attached to any pair $I,J$ of monomial ideals of finite colength of the ring of analytic function germs $(\C^n,0)\to \C$.
In particular, we obtain a combinatorial expression for this sequence when $J$ is diagonal. We also show
several relations of $\LL^*_J(I)$ with other numerical invariants associated to $I$ and $J$.
\end{abstract}

\maketitle


\section{Introduction}\label{intro}



The multiplicity and the \L ojasiewicz exponent of ideals of finite colength in a Noetherian local ring
are fundamental numerical invariants that have numerous applications in commutative algebra, algebraic geometry and singularity theory
(see for instance \cite{LT, Ploski, Cargese}).

The notion of multiplicity of ideals in a Noetherian local ring was extended to sequences of ideals $(I_1,\dots, I_n)$ of finite colength by
Risler and Teissier in \cite{Cargese}. This notion was further developed by Rees in his article \cite{Rees2}, where
he also introduced the fundamental notion of joint reduction.
Moreover, Swanson gave in \cite{Swanson} a version of the Rees' multiplicity theorem for mixed multiplicities.

\L ojasiewicz exponents were initially introduced in the context of complex analytic geometry.
Due to the fundamental work of Lejeune and Teissier \cite{LT}, \L ojasiewicz exponents admit an equivalent formulation in terms of the
notion of integral closure of ideals. Consequently, these numbers have a translation in terms of multiplicities of ideals, by virtue of the Rees'
multiplicity theorem (see relation \ref{LJImult}).

Let $(R,\m)$ be a Noetherian local ring of dimension $n$. In \cite{BiviaMRL} we considered an extension of
 the notion of mixed multiplicity of $n$
ideals of finite colength of $R$ to certain sequences of ideals
$(I_1,\dots, I_n)$ that are not assumed to have finite colength. We call this number the Rees' multiplicity of $(I_1,\dots, I_n)$
and we denote it by $\sigma(I_1,\dots, I_n)$. Analogous to the idea of extending the notion of Samuel multiplicity of ideals to sequences
of $n$ ideals in a ring of dimension $n$, in \cite{BiviaMZ2} we started the task of developing a similar idea for \L ojasiewicz exponents
of ideals. Hence, if $(I_1,\dots, I_n)$ is a sequence of ideals of $R$ for which
$\sigma(I_1,\dots,I_n)<\infty$ and if $J$ is a proper ideal of $R$, then we introduced the notion of mixed
\L ojasiewicz exponent of $(I_1,\dots, I_n)$ with respect to $J$ (see Definition \ref{mixedLoj}). This number is denoted by $\LL_J(I_1,\dots, I_n)$.

Let $\O_n$ be the ring of analytic function germs $(\C^n,0)\to \C$ and let $\m_n$ be the maximal ideal of $\O_n$.
In \cite{BiviaMZ2} we addressed the problem of finding an effective
procedure to compute $\LL_J(I_1,\dots, I_n)$ in the case where $R=\O_n$, the ideals $I_1,\dots, I_n$ are generated by monomials
and $J=\m_n$. In \cite{BE1, BENewton} we considered the problem of determining $\LL_J(I_1,\dots, I_n)$ for ideals
in $\O_n$ with the aid of a fixed Newton filtration.


By \cite[Corollary 3.8]{BF1}, the following relation between multiplicities and \L ojasiewicz exponents holds:
\begin{equation}\label{eIeJLJI}
\frac{e(I)}{e(J)}\leq \LL_J^{(1)}(I)\cdots \LL_J^{(n)}(I),
\end{equation}
where $I$ and $J$ are ideals of finite colength of $R$
and $\LL_J^{(i)}(I)=\LL_J(I,\dots, I, J,\dots, J)$,
with $I$ repeated $i$ times and $J$ repeated $n-i$ times.
We denote the vector $(\LL_J^{(n)}(I),\dots, \LL_J^{(1)}(I))$ by $\LL_J^*(I)$.
In \cite{Hickel}, Hickel proved inequality (\ref{eIeJLJI}), by using different techniques,
in the case where $R$ is an equicharacteristic regular local ring and
$J$ equals the maximal ideal. In this context, he also characterized the class of ideals
$I$ for which equality holds in (\ref{eIeJLJI}) when $n=2$ (see \cite[Proposition  5.1]{Hickel}).

Let us denote $\LL^{(i)}_{\m_n}(I)$ by $\LL^{(i)}_{0}(I)$, for all $i=1,\dots, n$.
In \cite[Theorem 3.5]{BiviaBAMS} we
proved that, if $I$ is a monomial ideal of $\O_n$, then $e(I)=\LL_0^{(1)}(I)\cdots \LL_0^{(n)}(I)$ if and only if
there exist homogeneous polynomials $g_1,\dots, g_n\in\C[x_1,\dots, x_n]$ such that $\overline{I}=\overline{\langle g_1,\dots, g_n\rangle}$.
That is, we characterized the equality in (\ref{eIeJLJI}) when $I$ is a monomial ideal and $J=\m_n$.


Let $I$ and $J$ be monomial ideals of $\O_n$ of finite colength.
The main purpose of this article is to compute the sequence $\LL^*_J(I)$
in terms of the combinatorial information supplied by the respective Newton polyhedra of $I$ and $J$.
We have obtained an upper bound for each $\LL_J^{(i)}(I)$ which becomes an equality when $J$ is diagonal,
that is, when $J$ is of the form $J=\overline{\langle x_1^{a_1},\dots, x_n^{a_n}\rangle}$, for some $a_1,\dots, a_n\in\Z_{\geq 1}$,
where the bar denotes integral closure. The results of the article are mainly motivated by the problem of characterizing the equality in (\ref{eIeJLJI}),
by \cite[Theorem 5.5]{BiviaeIeJ} and the results of \cite{Hickel}.
Next we describe more precisely the structure and contents of the article.


In Section \ref{preliminars} we recall some notions and results needed to expose our work.
Hence we recall basic notions like Newton filtration, $J$-non-degenerate sequence of ideals and $J$-non-degenerate map,
where $J$ denotes a fixed monomial ideal of finite colength of $\O_n$ (these notions generalize the notion of semi
weighted-homogeneous map $(\C^n,0)\to (\C^n,0)$).
Section \ref{mixedL} is devoted to developing general results about the sequence $\LL^*_J(I)$, when $I$ and $J$ are arbitrary
ideals of a Noetherian local ring $R$. In particular, in Proposition \ref{decreix} we prove that $\LL^*_J(I)$ forms a decreasing sequence provided that $I\subseteq \overline J$. We also introduce and characterize the class of {\it Hickel ideals with respect to $J$} (see Corollary \ref{HickelwrtJ}).

In Section \ref{Lstar} we explore the sequence $\LL^*_J(I)$ when the ideal $I$ is generated by the components of a $J$-non-degenerate map
$(\C^n,0)\to (\C^n,0)$ and $J$ is a monomial ideal of $\O_n$ of finite colength.
We obtain an expression for this sequence when $I\subseteq \overline J$ (see Theorem \ref{LstarJnodeg})
and derive a characterization of $J$-non-degenerate maps in terms of $\LL_J^*(I)$ (see Corollary \ref{Calp}).

Let $I$ and $J$ be arbitrary monomial ideals of $\O_n$ of finite colength. The main result of
Section \ref{Lstarmonomial} is Theorem \ref{essencial},
where we show how to obtain an upper bound for the elements of the sequence
$\LL_J^*(I)$ from any $J$-non-degenerate sequence
of ideals contained in $I$.

Section \ref{applications} is devoted to showing two applications of Theorem \ref{essencial}. In \cite{BiviaeIeJ}
we constructed a particular $J$-non-degenerate sequence $(K_1,\dots K_n)$ of ideals contained in $I$. So we apply
this special sequence of ideals to derive an upper bound for the numbers $\LL_J^{(i)}(I)$ in
terms of the Newton filtration of $J$ and the intersection with $\Gamma_+(I)$
of the half rays determined by the vertices of $\Gamma_+(J)$.
The second application deals with the case where $J$ is diagonal. In this case we show that the mentioned upper bounds actually coincide with the numbers
$\LL_J^{(i)}(I)$.
As a corollary, we obtain that if $J$ is diagonal, then
equality holds in (\ref{eIeJLJI}) if and only if there exists some $s\geq 1$ such that
$\overline{I^s}=\overline{\langle g_1,\dots, g_n\rangle}$, where $(g_1,\dots, g_n)$ is $J$-non-degenerate
(see Corollary \ref{corolfinalqh}). 



\section{Preliminary concepts}\label{preliminars}

\subsection{Newton filtrations}
In this section we show some combinatorial definitions that we need in order to expose our results. These definitions
already appear in \cite[Section 4]{BiviaeIeJ} and \cite[Section 3]{BENewton}. For the sake of completeness we include some of them also here.

Let $\O_n$ denote the ring of analytic function germs $(\C^n,0)\to \C$.
Let us fix coordinates $(x_1,\dots, x_n)$ in $\C^n$. If $k\in\Z^n_{\geq 0}$, then we denote the monomial
$x_1^{k_1}\cdots x_n^{k_n}$ by $x^k$. We say that a proper ideal $I$ of $\O_n$ is {\it monomial} when $I$ admits a generating system formed by monomials.
Let $h\in\O_n$ and let $h=\sum_ka_kx^k$ be the Taylor expansion of $h$ around the origin. The
{\it support of $h$}, denoted by $\supp(h)$, is the set $\{k\in \Z^n_{\geq 0}: a_k\neq 0\}$.
If $\Delta$ is any subset of $\R^n_{\geq 0}$, then we denote by $h_\Delta$ the sum of those terms $a_kx^k$ such that $k\in\supp(h)\cap \Delta$.
We set $h_\Delta=0$ whenever $\supp(h)\cap \Delta=\emptyset$.
Given an ideal $I$ of $\O_n$, the {\it support of $I$}, denoted by $\supp(I)$,
is defined as the union of the supports of the elements of $I$.

If $A\subseteq \Z^n_{\geq 0}$, $A\neq\emptyset$, then the {\it Newton polyhedron determined by $A$} is the set
$\Gamma_+(A)$ obtained as the convex hull of $\{k+v: k\in A,\, v\in\R^n_{\geq 0}\}$.
If $\Gamma_+$ is a subset of $\R^n_{\geq 0}$ such that $\Gamma_+=\Gamma_+(A)$, for some $A\subseteq \Z^n_{\geq 0}$, then we will say that
$\Gamma_+$ is a {\it Newton polyhedron}.

Given an element $h\in\O_n$, $h\neq 0$, the {\it Newton polyhedron of $h$}
is $\Gamma_+(h)=\Gamma_+(\supp(h))$. If $h=0$, then we set $\Gamma_+(h)=\emptyset$.
Analogously, given a non-zero ideal $I\subseteq \O_n$, the {\it Newton polyhedron of $I$}, is defined as
$\Gamma_+(I)=\Gamma_+(\supp(I))$. It is known that if $I$ is a monomial ideal of $\O_n$, then
the integral closure of $I$ is generated by those monomials $x^k$ such that $k\in\Gamma_+(I)$ (see for instance \cite[Proposition 1.4.6]{HS}
or \cite[Proposition 3.4]{T2}).

Let us fix a Newton polyhedron $\Gamma_+\subseteq\R^n$. We say that $\Gamma_+$ is {\it convenient} when $\Gamma_+$ meets each
coordinate axis in a point different from the origin. In particular, if $J$ is a proper ideal of $\O_n$ of finite colength, then
$\Gamma_+(J)$ is convenient.

If $v\in\R^n_{\geq 0}$, then we define $\ell(v,\Gamma_+)=\min\{\langle v,k\rangle: k\in\Gamma_+\}$, where $\langle\,,\rangle$
denotes the standard scalar product in $\R^n$. We also set $\Delta(v,\Gamma_+)=\{k\in\Gamma_+: \langle v,k\rangle=\ell(v,\Gamma_+)\}$.
We say that a subset $\Delta\subseteq \Gamma_+$ is a {\it face} of $\Gamma_+$ when there exists some $v\in\R^n_{\geq 0}$ such that
$\Delta(v,\Gamma_+)=\Delta$. In this case we say that $v$ {\it supports} $\Delta$. If $\Delta$ is a face of $\Gamma_+$, then the {\it dimension of $\Delta$} is the minimum of the dimensions
of the affine subspaces of $\R^n$ containing $\Delta$. The faces of $\Gamma_+$ of dimension $0$ or $n-1$ will be called
{\it vertices} or {\it facets}, respectively. We denote by $\v(\Gamma_+)$ the set of vertices of $\Gamma_+$.
 The union of all compact faces of $\Gamma_+$ will be denoted by
$\Gamma$ and we will refer to this subset of $\Gamma_+$ as the {\it Newton boundary of $\Gamma_+$}.

We say that a given vector $v\in\Z^n_{\geq 0}$, $v\neq 0$, is {\it primitive}, when the non-zero components of $v$ are mutually prime
integers. We denote by $\F(\Gamma_+)$ the set of primitive vectors of $\Z^n_{\geq 0}$ supporting some facet of $\Gamma_+$.
Let us denote by $\F_c(\Gamma_+)$ the set of those $v\in\F(\Gamma_+)$ such that $\Delta(v,\Gamma_+)$
is compact and by $\F'(\Gamma_+)$ the set of those $v\in\F(\Gamma_+)$ such that $\ell(v,\Gamma_+)>0$. Obviously
$\F_c(\Gamma_+)\subseteq \F'(\Gamma_+)$ and equality holds if $\Gamma_+$ is convenient.

Let us suppose that $\F'(\Gamma_+)=\{v^1, \dots, v^s\}$, for some primitive vectors $v^1,\dots, v^s\in\Z^n_{\geq 0}$, $s\geq 1$.
Let $M_\Gamma$ denote the least common multiple of $\{\ell(v^1,\Gamma_+),\dots, \ell(v^s,\Gamma_+)\}$. We define the
{\it filtrating map associated to $\Gamma_+$} as the map
$\phi_\Gamma:\R^n_{\geq 0}\to \R_{\geq 0}$ given by
$$
\phi_\Gamma(k)=\min\bigg\{\frac{M_\Gamma}{\ell(v^i,\Gamma_+)}\langle
k, v^i\rangle: i=1,\dots, s\bigg\},\quad \textnormal{for all
$k\in\R^n_{\geq 0}$}.
$$


If $\Delta$ is any subset of $\R^n$, then we denote by $C(\Delta)$ the {\it cone
over $\Delta$}, that is, the union of all half-lines emanating from the origin and passing through some point of $\Delta$.
It is easy to check that $\phi_\Gamma(\Z^n_{\geq 0})\subseteq\Z^n_{\geq 0}$,
$\phi_\Gamma(k)=M_\Gamma$, for all $k\in\Gamma$, and the map $\phi_\Gamma$ is linear on each
cone $C(\Delta)$, where $\Delta$ is any compact face of $\Gamma_+$.
Therefore, we define the map $\nu_\Gamma:\O_n\to \R_{\geq 0}\cup \{+\infty\}$ by
$
\nu_\Gamma(h)=\min\big\{\phi_\Gamma(k): k\in\supp(h)\big\}
$
for any $h\in\O_n$, where we set $\nu_\Gamma(0)=+\infty$.

For any $r\in \Z_{\geq 0}$, let us consider the ideal
\begin{equation}\label{elsBr}
\B_r=\big\{h\in\O_n: \nu_\Gamma(h)\geq r\big\}\cup\{0\}.
\end{equation}
Obviously $\B_{r+1}\subseteq \B_{r}$, for all $r\in\Z_{\geq 0}$.
Thus $\{\B_r\}_{r\geq 0}$ is a decreasing sequence of ideals.
We will indistinctly refer to the map $\nu_\Gamma$ or to the family of ideals $\{\B_r\}_{r\geq 0}$ as the {\it
Newton filtration} induced by $\Gamma_+$ (see also \cite{BFS, K}). This notion generalizes the notion of weighted homogeneous filtration of $\O_n$.


\subsection{Mixed multiplicities and $J$-non-degeneracy of sequences of ideals}\label{M&NF}

Along this section we will suppose that $(R,\m)$ is a Noetherian local ring of dimension $n$.
If $I$ is an ideal of $R$, then we denote by $\overline I$ the integral closure of $I$ and, if $I$ has finite colength, then
$e(I)$ will denote the Samuel multiplicity of $I$ (see \cite{HS, Matsumura,V}).

Given $g_1,\dots, g_r\in R$, if these elements generate an ideal of finite colength of $R$, then we will also write
$e(g_1,\dots, g_r)$ instead of $e(\langle g_1,\dots, g_r\rangle)$.

If $I_1,\dots, I_n$ are ideals of $R$ of finite colength,
then we denote by $e(I_1,\dots, I_n)$ the mixed multiplicity
of $I_1,\dots, I_n$ defined by Teissier and Risler in \cite[\S 2]{Cargese}.
We also refer to \cite[\S 17.4]{HS}, \cite{Rees2} or \cite{Swanson} for the definition and fundamental properties of mixed multiplicities of ideals.


Given two ideals $I$ and $J$ of $R$ of finite colength and an integer
$i\in \{1,\dots, n\}$, we define
\begin{equation}\label{eijIJ}
e_i(I,J)=e(I,\dots, I, J,\dots, J),
\end{equation}
where $I$ is repeated $i$ times and $J$ is repeated $n-i$ times.

\begin{defn}\label{lasigma}\cite{BiviaMRL} Let $I_1,\dots,
I_n$ be ideals of $R$. If the set of natural numbers $\{e(I_1+\m^r,\dots,
I_n+\m^r): r\in\Z_{\geq 1}\}$ is bounded, then we define the {\it Rees' mixed multiplicity
of $I_1,\dots, I_n$} as
\begin{equation}\label{sigma}
\sigma(I_1,\dots, I_n)=\max_{r\in\Z_{\geq 1}}\, e(I_1+\m^r,\dots,
I_n+\m^r).
\end{equation}
Otherwise, we set $\sigma(I_1,\dots, I_n)=\infty$.
\end{defn}

Let us suppose that the residue field $\k=R/\m$ is
infinite. Let $I_1,\dots, I_r$ be proper ideals of $R$ and let us
fix a generating system $a_{i1},\dots, a_{is_i}$ of $I_i$,
for all $i=1,\dots, r$. Let $s=s_1+\cdots +s_r$. We say that a given
property holds for {\it sufficiently general elements of
$I_1\oplus \cdots \oplus I_r$} if there exists a non-empty
Zariski-open set $U$ in $\k^s$ such that all elements $(g_1,\dots,
g_r)\in I_1\oplus \cdots \oplus I_r$ satisfy the said property provided that
\begin{enumerate}
\item[(a)] for all $i=1,\dots, r$: $g_i=\sum_{j}u_{ij}a_{ij}$, where $u_{ij}\in R$, for all $j=1,\dots, s_i$, and
\item[(b)] the image of $(u_{11}, \dots,u_{1s_1},\dots, u_{r1}, \dots, u_{rs_r})$ in $\k^s$ belongs to $U$.
\end{enumerate}

\begin{prop}\label{Prop2.9}\textnormal{\cite[2.9]{BiviaMRL}}
Let $(R,\m)$ be a Noetherian local ring of dimension $n$ with infinite residue field.
Let $I_1,\dots, I_n$ be proper ideals of $R$. Then, $\sigma(I_1,\dots, I_n)<\infty$ if and only if
there exist elements $g_i\in I_i$, for $i=1,\dots, n$, such that
$\langle g_1,\dots, g_n\rangle$ has finite colength. If $\sigma(I_1,\dots, I_n)<\infty$, then
$\sigma(I_1,\dots, I_n)=e(g_1,\dots, g_n)$ for
sufficiently general elements $(g_1,\dots, g_n)\in I_1\oplus
\cdots \oplus I_n$.
\end{prop}

We remark that the case of Proposition \ref{Prop2.9} where $I_1,\dots, I_n$ have finite colength
follows as a consequence of the theorem of existence of joint reductions (see \cite[p.\,336]{HS} or \cite[p.\,4]{Swanson}).

Along the rest of this section, we will suppose that $J$ is a monomial ideal of $\O_n$ of finite colength.
We denote by $\nu_J$ the Newton filtration induced by $\Gamma_+(J)$ and by
$\phi_J$ the corresponding filtrating map. Let us also set $M_J=M_{\Gamma(J)}$, where $\Gamma(J)$
denotes the Newton boundary of $\Gamma_+(J)$.
If $I$ is a non-zero ideal of $\O_n$, then we define $\nu_J(I)=\min\{\nu_J(h): h\in I\}$.

For instance, if $J=\m_n$, then $\phi_{\m_n}(k)=\vert k\vert$, for all $k\in\R^n_{\geq 0}$, where $\vert (k_1,\dots, k_n)\vert=
k_1+\cdots+k_n$, for any $(k_1,\dots, k_n)\in \R^n_{\geq 0}$.

Let $(I_1,\dots, I_n)$ be an $n$-tuple of proper ideals such that
$\sigma(I_1,\dots, I_n)<\infty$. In general, we have that
\begin{equation}\label{desref}
\sigma(I_1,\dots, I_n)\geq \frac{\nu_J(I_1)\cdots\nu_J(I_n)}{M_J^n}e(J)
\end{equation}
(see \cite[Proposition 3.2]{BENewton}).
We say that $(I_1,\dots, I_n)$ is {\it $J$-non-degenerate} when equality holds in (\ref{desref})
(see \cite[Definition 3.3]{BENewton}).

Let $g=(g_1,\dots, g_n):(\C^n,0)\to (\C^n,0)$ be an analytic map germ.
We say that $g$ is {\it $J$-non-degenerate} when the $n$-tuple of ideals $(\langle g_1\rangle,\dots, \langle g_n\rangle)$ is $J$-non-degenerate.
That is, when
\begin{equation}\label{gJnodeg}
e(g_1,\dots, g_n)=\frac{\nu_J(g_1)\cdots \nu_J(g_n)}{M_J^n}e(J).
\end{equation}
In Theorem \ref{thm3.3BFS} we recall a characterization of this class of maps given in \cite{BFS}.

Let $h\in\O_n$, $h\neq 0$, and suppose that $h=\sum_ka_kx^k$ is the Taylor expansion of $h$
around the origin. If $\Delta$ is a compact face of $\Gamma_+(J)$, then we denote by $\p_{J, \Delta}(h)$ the sum of all terms $a_kx^k$
such that $k\in C(\Delta)$ and $\nu_J(x^k)=\nu_J(h)$. If no such terms exist, then we set $\p_{J, \Delta}(h)=0$.

\begin{thm}\cite[3.3]{BFS}\label{thm3.3BFS}
Let $g=(g_1,\dots, g_n):(\C^n,0)\to (\C^n,0)$ be an analytic map germ such that $g^{-1}(0)=\{0\}$.
Then, the following conditions are equivalent:
\begin{enumerate}
\item[(a)] $g$ is $J$-non-degenerate;
\item[(b)] the set germ at $0$ of common zeros of $\p_{J,\Delta}(g_1),\dots, \p_{J,\Delta}(g_n)$
is contained in $\{x\in\C^n: x_1\cdots x_n=0\}$, for all compact faces $\Delta$ of $\Gamma_+(J)$.
\end{enumerate}
\end{thm}

Under the hypothesis of Theorem \ref{thm3.3BFS}, let us assume that
$\nu_J(g_1)=\cdots=\nu_J(g_n)=M_J$. Hence $\p_{J,\Delta}(g_i)=(g_i)_\Delta$, for all $i=1,\dots, n$.
Therefore, in this case, $g$ is $J$-non-degenerate if and only if
$e(g_1,\dots, g_n)=e(J)$, which is to say that $\langle g_1,\dots, g_n\rangle$ is a reduction of $J$,
by the Rees' multiplicity theorem \cite[p.\,222]{HS}.

\begin{rem}\label{NNDideals}
We recall that reductions of monomial ideal ideals are characterized in \cite[Proposition 3.6]{BiviaMRL}. These are the so called
{\it Newton non-degenerate ideals} (see \cite{BiviaMRL,BFS,S,T2}).
Let $I=\langle g_1,\dots, g_s\rangle$ be an ideal of $\O_n$. Then,
$I$ is called
Newton non-degenerate when the set germ at $0$ of $\{x\in\C^n: (g_1)_\Delta(x)=\cdots=(g_s)_\Delta(x)=0\}$
is contained in $\{x\in\C^n: x_1\cdots x_n=0\}$, for all compact faces $\Delta$ of $\Gamma_+(I)$
(it is immediate to check that this definition does not depend on the chosen generating
system of $I$). This kind of ideals was originally introduced by Saia in \cite{S} motivated by the notion of Newton non-degenerate function (see \cite{K}).
\end{rem}

As we see in the next result, if $g$ is a $J$-non-degenerate map, then the sequence of mixed multiplicities $e_i(I(g), J)$, $i=0,\dots, n$, can also be expressed in terms of $\nu_J$, where $I(g)$ denotes the ideal of $\O_n$ generated by the component functions of $g$.

\begin{prop}\label{eideg}\cite{BiviaeIeJ}
Let $g=(g_1,\dots, g_n):(\C^n,0)\to (\C^n,0)$ be a $J$-non-degenerate map.
Let $d_i=\nu_J(g_i)$, for all $i=1,\dots, n$. Let us suppose that $d_1\leq \cdots \leq d_n$. Then
$$
e_i(I(g), J)=\frac{d_1\cdots d_i}{M^i}e(J).
$$
\end{prop}



The next result shows a characterization of the $J$-non-degeneracy of $n$-tuples of ideals.

\begin{prop}\label{caractwrtJ}\cite{BiviaeIeJ}
Let $I_1,\dots, I_n$ be ideals of $\O_n$ such that $\sigma(I_1,\dots, I_n)<\infty$.
Then, $(I_1,\dots, I_n)$ is $J$-non-degenerate if and only if there exist $a_1,\dots, a_n, d\in\Z_{\geq 1}$ such that
$\sigma(I_1^{a_1},\dots, I_n^{a_n})=e(J^d)$ and $\nu_J(I_1^{a_1})=\cdots=\nu_J(I_n^{a_n})=dM_J$.
\end{prop}

We remark that the previous result has been our motivation to introduce in \cite[Definition 7]{BiviaeIeJ} the notion
of $J$-non-degeneracy of a sequence of elements $g_1,\dots, g_n$ in an arbitrary local ring $(R,\m)$, where $J$ denotes
any proper ideal of $R$. We will apply the following result in Section \ref{Lstar}.

\begin{cor}\label{Jpasseja}
Let $J$ be a monomial ideal of $\O_n$ of finite colength.
Let $(I_1,\dots, I_n)$ be a $J$-non-degenerate $n$-tuple of ideals of $\O_n$.
Then, $(I_1,\dots, I_{i-1}, J, I_{i+1},\dots, I_n)$ is $J$-non-degenerate, for all $i=1,\dots, n$.
\end{cor}

\begin{proof}
Let us suppose, without loss of generality, that $i=1$. Let $M=M_J$. By Proposition \ref{caractwrtJ}, there exist
$a_1,\dots,a_n,d\in\Z_{\geq 1}$ such that $e(I_1^{a_1},\dots, I_n^{a_n})=e(J^d)$ and $\nu_J(I_1^{a_1})=\cdots=\nu_J(I_n^{a_n})=dM$.
In particular, the condition $\nu_J(I_1^{a_1})=dM$ implies that
$
I_1^{a_1}\subseteq \overline{J^{d}}.
$

Therefore
\begin{align}
e(J^d)=e(I_1^{a_1},\dots, I_n^{a_n})&\geq e(\overline{J^{d}},I_2^{a_2},\dots, I_n^{a_n}) \nonumber\\
&\geq \frac{dM\cdots dM}{M^n}e(J)=d^ne(J)=e(J^d),\label{ineq2}
\end{align}
where we have applied (\ref{desref}) in the first inequality of (\ref{ineq2}). Hence the result follows by applying
Proposition \ref{caractwrtJ}.
\end{proof}

\section{Mixed \L ojasiewicz exponents and Hickel ideals}\label{mixedL}

Let $J$ and $I$ be proper ideals of $\O_n$. Let us suppose that
$\{f_1,\dots, f_p\}$ is a generating system of $J$ and
$\{g_1,\dots, g_q\}$ is a generating system of $I$. Let us consider the maps
$f=(f_1,\dots, f_p):(\C^n,0)\to (\C^p,0)$ and $g=(g_1,\dots, g_q):(\C^n, 0)\to (\C^q,0)$.
The {\it \L ojasiewicz exponent of $I$ with respect to $J$}, denoted by $\LL_J(I)$, is defined
as the infimum of the set of those $\alpha\in \R_{\geq 0}$ for which there exists a constant $C>0$ and an open neighbourhood $U$ of $0\in\C^n$
such that
\begin{equation}\label{Lojanalytic}
\Vert f(x)\Vert^\alpha \leq C\Vert g(x)\Vert,\quad\text{for all $x\in U$.}
\end{equation}
When the set of such $\alpha$ is empty, then we fix $\LL_J(I)=+\infty$.
It is known that $\LL_J(I)$ exists if and only if $\V(I)\subseteq \V(J)$ and, if this is the case, then $\LL_J(I)$ is a rational number
(see \cite[Théorème 4.6]{LT} or \cite{Tresonances}). Moreover, in \cite[Théorème 7.2]{LT} the \L ojasiewicz exponent of $I$
with respect to $J$ is characterized as follows:
\begin{equation}\label{LJIintegral2}
\LL_J(I)=\inf\left\{\frac rs: r,s\in\Z_{\geq 1},\,J^r\subseteq \overline{I^s}    \right\}.
\end{equation}
Therefore $J\subseteq \overline I$ if and only if $\LL_J(I)\leq 1$.


\begin{rem}\label{LJgeq1} \begin{enumerate}
\item[(a)] Let us suppose that $\V(I)=\V(J)$. By (\ref{LJIintegral2}) it is immediate to see that $\LL_J(I)\LL_I(J)\geq 1$.
Hence, if $J\subseteq \overline I$, then $\LL_I(J)\geq 1$.

\item[(b)] We recall that relation (\ref{LJIintegral2}) constitutes the definition of {\it \L ojasiewicz exponent of $I$ with respect to $J$} whenever
$I$ and $J$ are ideals of an arbitrary Noetherian local ring such that $\sqrt{J}\subseteq \sqrt{I}$.
\end{enumerate}
\end{rem}

Let $g\in \O_n$, $g\neq 0$. We define the {\it order of $g$}, denoted by $\ord(g)$, as the maximum of those $r\in\Z_{\geq 0}$
for which $g\in\m_n^r$. We set $\ord(0)=+\infty$. If $I$ is an ideal of $\O_n$, then we define the {\it order of $I$} as
$\ord(I)=\min\{\ord(g): g\in I\}=\max\{r\in\Z_{\geq 0}: I\subseteq \m_n^r\}$.

Let $\varphi=(\varphi_1,\dots, \varphi_n):(\C,0)\to (\C^n,0)$ be an analytic curve.
We define the {\it order of $\varphi$} as $\ord(\varphi)=\min\{ \ord(\varphi_1),\dots, \ord(\varphi_n)\}$.
Let $\varphi^*:\O_n\to \O_1$ be the ring morphism given by $g\mapsto g\circ \varphi$, for all $g\in\O_n$.
If $I=\langle g_1,\dots, g_s\rangle$ is a non-zero ideal of $\O_n$, then
$\ord(\varphi^*(I))=\min\{\ord(g\circ \varphi):g\in I\}
=\min\{\ord(g_1\circ\varphi),\dots,\ord(g_s\circ\varphi)\}$.

We will apply the following result in Section \ref{applications}.

\begin{lem}\cite{LT}\label{byarcs}
Let $I$ and $J$ be proper ideals of $\O_n$ such that $\V(I)\subseteq \V(J)$. Then
$$
\LL_J(I)=\sup_{\varphi\in \Omega}\frac{\ord(\varphi^*(I))}{\ord(\varphi^*(J))},
$$
where $\Omega$ denotes the set of non-zero analytic curves $\varphi:(\C,0)\to (\C^n,0)$.
\end{lem}

We recall that in the Rees' multiplicity theorem the quasi-unmixedness condition on the given ring is required (see \cite[p.\,222]{HS}).
This condition is also known as formal equidimensionality (see \cite[p.\,401]{HS} or
\cite[p.\,251]{Matsumura}). Moreover, by \cite[p.\,149]{HIO},
if $(R,\m)$ is a Noetherian local ring, then $R$ is quasi-unmixed
if and only if the equality $\overline J=\overline I$ holds, for any
pair of ideals $J$ and $I$ of $R$ such that $J\subseteq I$ and $e(I)=e(J)$.

In the remaining section, we denote by $(R,\m)$ a Noetherian quasi-unmixed local ring of dimension $n$.
Let us fix two integers $p,q\in\Z_{\geq 1}$ and let $I$ and $J$ be ideals of $R$.
Then, we have the following equivalences:
$$
J^r\subseteq \overline{I^s}\Longleftrightarrow
\overline{I^s}=\overline{I^s+J^r}\Longleftrightarrow e(I^s)=e(I^s+J^r).
$$

Therefore, by (\ref{LJIintegral2}), we can write $\LL_J(I)$ as follows:
\begin{equation}\label{LJImult}
\LL_J(I)=\inf\left\{\frac rs: r,s\in\Z_{\geq 1},\,e(I^s)=e(I^s+J^r)    \right\}.
\end{equation}

As recalled in Section \ref{M&NF},
the notion of Samuel multiplicity of an ideal $I$ of $R$ of finite colength was extended to $n$-tuples
$(I_1,\dots, I_n)$ of ideals of finite colength by Teissier and Risler in \cite{Cargese}.
Analogously, applying the notion of Rees' mixed multiplicity (Definition \ref{lasigma})
and relation (\ref{LJImult}), we extended the notion of \L ojasiewicz exponent $\LL_J(I)$ to $n$-tuples of ideals $\LL_J(I_1,\dots, I_n)$
(see \cite{BiviaMZ2} and \cite{BENewton}).

Let $I_1,\dots, I_n$ be ideals of $R$ such that $\sigma(I_1,\dots, I_n)<\infty$ and let $J$
be a proper ideal of $R$.
We define
\begin{equation}\label{erreJ}
r_J(I_1,\dots, I_n)=\min\big\{r\in\Z_{\geq 1}: \sigma(I_1,\dots,
I_n)=\sigma(I_1+J^r,\dots, I_n+J^r)\big\}.
\end{equation}
When $J=\m$, we denote $r_\m(I_1,\dots, I_n)$ simply by $r(I_1,\dots, I_n)$
.

\begin{defn}\label{mixedLoj}
Let $I_1,\dots, I_n$
be ideals of $R$ such that $\sigma(I_1,\dots, I_n)<\infty$. Let $J$ be a proper ideal of $R$.
The {\it \L ojasiewicz exponent of $I_1,\dots, I_n$ with respect to $J$}, denoted by
$\LL_J(I_1,\dots, I_n)$, is defined as
\begin{equation}\label{defLoj}
\LL_J(I_1,\dots, I_n)=\inf_{s\geq 1}\frac{r_J(I_1^s,\dots, I_n^s)}{s}.
\end{equation}
\end{defn}

By applying (\ref{erreJ}), we have
\begin{equation}\label{LJmillor}
\LL_J(I_1,\dots, I_n)=\inf\left\{ \frac rs: r,s\in\Z_{\geq 1},\, \sigma(I_1^s,\dots, I_n^s)=\sigma(I_1^s+J^r,\dots, I_n^s+J^r)\right\}.
\end{equation}

Let $I_1,\dots, I_i, J, I$ be ideals of $R$, where $i\in \{1,\dots, n-1\}$.
If there is no risk of confusion, when we write $\sigma(I_1,\dots, I_i, J,\dots, J)$ or $\LL_I(I_1,\dots, I_i, J,\dots, J)$,
we will tacitly assume that $J$ is repeated $n-i$ times, where we recall that $n=\dim(R)$.

Given ideals $I$ and $J$ of $R$ and an index $i\in\{1,\dots, n\}$, we denote by
$\LL_J^{(i)}(I)$ the \L ojasiewicz exponent $\LL_J(I,\dots, I, J,\dots, J)$, where $I$ is repeated $i$ times and $J$ is repeated $n-i$ times.
Thus $\LL_J^{(n)}(I)=\LL_J(I)$. We set $\LL_J^*(I)=(\LL_J^{(n)}(I),\dots,\LL_J^{(1)}(I))$.
We will denote $\LL_\m^{(i)}(I)$ and $\LL_\m^{*}(I)$ by $\LL_0^{(i)}(I)$ and $\LL_0^{*}(I)$, respectively, for all $i\in\{1,\dots, n\}$.

Let $g=(g_1,\dots, g_p):(\C^n,0)\to(\C^p,0)$ be an analytic map germ and let $J$ be an ideal of $\O_n$. If
$i\in\{1,\dots, n\}$, then we denote $\LL_J^{(i)}(\langle g_1,\dots, g_p\rangle)$ also by $\LL_J^{(i)}(g)$. We will also write
$\LL_0^*(g)$ instead of $\LL_0^*(\langle g_1,\dots, g_p\rangle)$.


\begin{prop}\label{lower} Let $J$ be a proper ideal of $R$.
For each $i=1,\dots, n$, let us consider ideals $I_i$ and $J_i$ of $R$
such that $I_i\subseteq J_i$ and $\sigma(I_1,\dots,
I_n)=\sigma(J_1,\dots,J_n)<\infty$. Then
\begin{equation}\label{monot}
\LL_J(I_1,\dots, I_n)\leqslant \LL_J(J_1,\dots, J_n),
\end{equation}
\end{prop}

\begin{proof}
It follows by replacing the maximal ideal $\m$ by $J$ in the proof of \cite[Proposition 4.7]{BiviaMZ2}.
\end{proof}

\begin{prop}\label{decreix}
Let $I$ and $J$ be ideals of $R$ of finite colength such that $I\subseteq \overline J$ and let $i\in\{1,\dots, n\}$, where $n=\dim(R)\geq 2$. Then
\begin{equation}\label{Lisimple}
\LL_J^{(i)}(I)=\inf\left\{\frac rs: r,s\in\Z_{\geq 1}, r\geq s,\,e_{i}(I^s+J^r, J)=e_{i}(I^s, J)\right\}.
\end{equation}
Moreover
\begin{equation}\label{LnL1}
\LL^{(n)}_J(I)\geq \cdots \geq \LL^{(1)}_J(I).
\end{equation}
\end{prop}

\begin{proof}
As observed in Remark \ref{LJgeq1}, the inclusion $I\subseteq \overline J$ implies that $\LL_J^{(n)}(I)=\LL_J(I)\geq 1$.
Hence the case $i=n$ comes from (\ref{LJImult}).
Let $i\in \{1,\dots, n-1\}$. Let us prove first that $\LL_J^{(i)}(I)\geq 1$. If $\LL_J^{(i)}(I)<1$,
by (\ref{defLoj}), there exists $r,s\in \Z_{\geq 1}$ such that $r<s$ and
\begin{equation}\label{IsJr}
e_{i}(I^s+J^r, J^s+J^r)=e_{i}(I^s, J^s).
\end{equation}
Since $I\subseteq \overline J$ and $r<s$, we have $I^s\subseteq {\overline{J}}^s\subseteq\overline{J^s}\subseteq \overline{J^r}$.
Therefore, the member on the left side of (\ref{IsJr}) is equal to $e_i(J^r, J^r)=e(J^r)=r^ne(J)$. Joining this with (\ref{IsJr}) and considering
the inclusion $I\subseteq\overline J$, we obtain that
$$
r^ne(J)=e_{i}(I^s, J^s)\geq e_{i}(J^s, J^s)=e(J^s)=s^ne(J).
$$
Then, $r\geq s$, which is a contradiction, since $r<s$. Hence $\LL_J^{(i)}(I)\geq 1$,
for all $i=1,\dots, n$.

Let us fix any $i\in\{1,\dots, n-1\}$. Since $\LL_J^{(i)}(I)\geq 1$, by applying (\ref{LJmillor}), we can write:
\begin{align}
\LL_J^{(i)}(I)&=\inf\left\{\frac rs: r,s\in\Z_{\geq 1}, r\geq s,\,e_{i}(I^s+J^r, J^s+J^r)=e_{i}(I^s, J^s)\right\} \nonumber\\
&=\inf\left\{\frac rs: r,s\in\Z_{\geq 1}, r\geq s,\,s^{n-i}e_{i}(I^s+J^r, J)=s^{n-i}e_{i}(I^s, J)\right\}   \nonumber\\
&=\inf\left\{\frac rs: r,s\in\Z_{\geq 1}, r\geq s,\,e_{i}(I^s+J^r, J)=e_{i}(I^s, J)\right\}.\nonumber
\end{align}
and then (\ref{Lisimple}) follows.

Let us prove \eqref{LnL1}. Let us fix $r,s\in\Z_{\geq 1}$ such that $r\geq s$ and $e_{i+1}(I^s+J^r, J)=e_{i+1}(I^s, J)$.
By virtue of the theorem of existence of superficial sequences (see for instance \cite[Proposition 17.2.2]{HS})
and \cite[Theorem 17.4.6]{HS} (see also \cite[p.\,306]{Cargese}), there exists a sufficiently general element
$(h_1,\dots, h_{n-i})\in J\oplus \cdots \oplus J$ such that, if $R_1=R/\langle h_1,\dots, h_{n-i}\rangle$ and $R_2=R/\langle h_1,\dots, h_{n-i-1}\rangle$, then
the following relations hold
\begin{align}
e_{i}(I^s+J^r, J)&=e((I^s+J^r)R_1)    &e_{i}(I^s, J)&=e(I^sR_1)  \label{R1}\\
e_{i+1}(I^s+J^r, J)&=e((I^s+J^r)R_2)    &e_{i+1}(I^s, J)&=e(I^sR_2).\label{R2}
\end{align}
As indicated after Lemma \ref{byarcs}, we assume that $R$ is quasi-unmixed. Therefore the quotient rings $R_1$ and $R_2$ are also quasi-unmixed, by
\cite[Proposition B.4.4]{HS}.
Hence, by the Rees' multiplicity theorem (see \cite[p.\,222]{HS}), we have that the condition $e_{i+1}(I^s+J^r, J)=e_{i+1}(I^s, J)$
and relation (\ref{R2}) imply that
\begin{equation}\label{I2JrR2}
\overline{(I^s+J^r)R_2}=\overline{I^sR_2}.
\end{equation}
Let $\pi:R_2\to R_1$ denote the natural projection. Taking $\pi$ to both sides of (\ref{I2JrR2})
and applying the persistence property of the integral closure of ideals (see \cite[p.\,2]{HS}), that is, the fact that the
image of the integral closure of a given ideal through a ring morphism is contained in the integral closure of
the image of the ideal, we conclude that
\begin{equation}\label{I2JrR1}
\overline{(I^s+J^r)R_1}=\overline{I^sR_1},
\end{equation}
which implies that $e((I^s+J^r)R_1)=e(I^sR_1)$, that is,
$e_i(I^s+J^r, J)=e_i(I^s, J)$, by \eqref{R1}. Therefore, by applying (\ref{Lisimple}), we conclude that
$\LL_J^{(i+1)}(I)\geq \LL_J^{(i)}(I)$.
\end{proof}

Let us consider in $\O_2$ the ideals $J=\m_2^2$ and $I=\m_2$. Then, we have that $\LL_J^{(2)}(I)=\frac 12$ and
$\LL_J^{(1)}(I)=1$. Hence we observe that the condition $I\subseteq \overline J$ can not be eliminated in (\ref{LnL1}).

Let $I$ and $J$ be ideals of $\O_n$ of finite colength.
By \cite[Corollary 3.8]{BF1} we know that
\begin{equation}\label{JHickel}
\frac{e(I)}{e(J)}\leq \LL^{(1)}_J(I)\cdots  \LL^{(n)}_J(I).
\end{equation}

We say that $I$ is {\it Hickel with respect to $J$} when equality holds in (\ref{JHickel}). If this condition holds when $J=\m_n$, then
we will simply say that $I$ is a {\it Hickel ideal} (see \cite[Section 2.2]{BF2}). The same notions are defined analogously for
analytic maps $g:(\C^n,0)\to (\C^p,0)$ such that $g^{-1}(0)=\{0\}$.


\begin{prop}\label{eiei-1}
Let $I$ and $J$ be ideals of $R$ of finite colength. Let $i\in\{1,\dots, n\}$. Then
$$
\frac{e_i(I,J)}{e_{i-1}(I,J)}\leq \LL_J^{(i)}(I).
$$
\end{prop}

\begin{proof}
By the theorem of existence of superficial sequences (see \cite[Proposition 17.2.2]{HS}),
there exists a sufficiently general element $(g_{i+1},\dots, g_n)\in J\oplus \cdots\oplus J$ such that
if $p:R\to R/\langle g_{i+1},\dots, g_n\rangle$ denotes the canonical projection, then
$e_{i}(I, J)=e(p(I))$ and $e_{i-1}(I,J)=e_{i-1}(p(I),p(J))$.
By \cite[Proposition 3.1]{BF1} we know that
$$
\frac{e\big(p(I)\big)}{e_{i-1}\big(p(I),p(J)\big)}\leq \LL_{p(J)}\big(p(I)\big).
$$
Then
\begin{equation}\label{seguidet}
\frac{e_i(I,J)}{e_{i-1}(I,J)}=\frac{e\big (p(I)\big)}{e_{i-1}\big(p(I),p(J)\big)}\leq \LL_{p(J)}\big(p(I)\big)\leq
\LL_J(\underbrace{I,\dots, I}_{i}, \underbrace{J,\dots, J}_{n-i})=\LL_J^{(i)}(I),
\end{equation}
where the second inequality of (\ref{seguidet}) follows from \cite[Proposition 3.6]{BF1}.
\end{proof}

As we see in the following lemma, the sequence $\LL_J^*(I)$ is determined by the sequence of mixed multiplicities
$e_0(I,J),e_1(I,J),\dots, e_n(I,J)$ when $I$ is Hickel with respect to $J$. The following result is analogous to \cite[Lemma 5.5]{BF2}.

\begin{cor}\label{HickelwrtJ}
Let $I$ and $J$ be ideals of $R$ of finite colength. Then
\begin{equation}\label{eiIJ}
\frac{e_i(I, J)}{e(J)}\leq \LL_J^{(1)}(I)\cdots\LL_J^{(i)}(I)
\end{equation}
for all $i=1,\dots, n$,
and the following conditions are equivalent:
\begin{enumerate}
\item[(a)] $I$ is Hickel with respect to $J$.
\item[(b)] $\frac{e_{i}(I,J)}{e(J)}=\LL^{(1)}_J(I)\cdots  \LL^{(i)}_J(I)$, for all $i=1,\dots, n$.
\item[(c)] $\LL^{(i)}_J(I)=\frac{e_{i}(I,J)}{e_{i-1}(I,J)}$, for all $i=1,\dots, n$.
\end{enumerate}
\end{cor}

\begin{proof}

By Proposition \ref{eiei-1} we have that
$$
\frac{e_i(I, J)}{e(J)}=\frac{e_{1}(I, J)}{e_{0}(I, J)} \frac{e_{2}(I, J)}{e_{1}(I, J)}
\cdots \frac{e_i(I, J)}{e_{i-1}(I, J)}\leq \LL_J^{(1)}(I)\cdots \LL_J^{(i)}(I).
$$
Thus (\ref{eiIJ}) follows.

Let us see the implication (a) $\Rightarrow$ (b). So, let us assume that $I$ is Hickel with respect to $J$. Then, we have the following inequalities
\begin{align*}
\LL_J^{(1)}(I)\cdots \LL_J^{(n-1)}(I)&\geq \frac{e_{n-1}(I,J)}{e(J)}=
\frac{e(I)}{e(J)}\frac{e_{n-1}(I,J)}{e(I)}  &&\textnormal{(by (\ref{eiIJ}))}  \nonumber   \\
&\geq
\LL_J^{(1)}(I)\cdots \LL_J^{(n)}(I)\frac{1}{\LL_J(I)}   &&\textnormal{(by Proposition \ref{eiei-1})}  \nonumber    \\
&=\LL_J^{(1)}(I)\cdots \LL_J^{(n-1)}(I).
\end{align*}
Hence
\begin{equation}\label{en-1eJ}
\frac{e_{n-1}(I,J)}{e(J)}=\LL_J^{(1)}(I)\cdots \LL_J^{(n-1)}(I).
\end{equation}
Using this equality, we similarly obtain that
\begin{align*}
\LL_J^{(1)}(I)\cdots \LL_J^{(n-2)}(I)&\geq \frac{e_{n-2}(I,J)}{e(J)}=
\frac{e_{n-2}(I, J)}{e_{n-1}(I,J)}\frac{e_{n-1}(I,J)}{e(J)}  &&\textnormal{(by (\ref{eiIJ}))}  \nonumber   \\
&\geq
\LL_J^{(1)}(I)\cdots \LL_J^{(n-1)}(I)\frac{1}{\LL_J^{(n-1)}(I)}   &&\textnormal{(by (\ref{en-1eJ}) and Proposition
 \ref{eiei-1})}  \nonumber    \\
&=\LL_J^{(1)}(I)\cdots \LL_J^{(n-2)}(I).
\end{align*}
Thus, by applying finite induction we obtain relation (b). The implication (b) $\Rightarrow$ (a) and the
equivalence between (b) and (c) are obvious.
\end{proof}




\section{The sequence $\LL_J^*(I)$ and $J$-non-degeneracy}\label{Lstar}

Along this section, we will suppose that $J$ is a monomial ideal of finite colength of $\O_n$.
Let $I_1,\dots, I_n$ be a family of $n$ ideals of $\O_n$ such that $\sigma(I_1,\dots, I_n)<\infty$ and let $I$ be another ideal of $\O_n$.
Then, the pair $(I;I_1,\dots, I_n)$ is said to be {\it $J$-linked} when there exists some $i_0$
for which $(I_1,\dots,I_{i_0-1}, I,I_{i_0+1},\dots, I_n)$ is $J$-non-degenerate
and $\nu_J(I_{i_0})=\max\{\nu_J(I_1),\dots, \nu_J(I_n)\}$.

\begin{thm}\label{mainBERevMC}\cite[3.11]{BENewton}
Let $I_1,\dots, I_n$ be ideals of $\O_n$ such that $\sigma(I_1,\dots, I_n)<\infty$.
Let $\{\B_r\}_{r\geq 0}$ be the Newton filtration induced by $\Gamma_+(J)$.
Let $r_i=\nu_J(J_i)$, for all $i=1,\dots, n$.
If $(I_1,\dots, I_n)$ is $J$-non-degenerate and $I$ is a proper ideal of $\O_n$, then
$$
\LL_I(I_1,\dots, I_n)\leq \LL_I(\B_{r_1},\dots, \B_{r_n}) \leq  \frac{\max\{\nu_J(J_1),\dots, \nu_J(J_n)\}}{\nu_J(I)}
$$
and the above inequalities turn into equalities if $(I; I_1,\dots, I_n)$ is $J$-linked.
\end{thm}

\begin{cor}\label{corolBE}
Let $(I_1,\dots, I_n)$ be a $J$-non-degenerate $n$-tuple of ideals of $\O_n$.
Then
\begin{equation}\label{encaixa}
\LL_J(I_1,\dots, I_n)=\frac{\max\{\nu_J(J_1),\dots, \nu_J(J_n)\}}{M_J}.
\end{equation}
\end{cor}

\begin{proof}
By Corollary \ref{Jpasseja}, it follows that $(J; I_1,\dots, I_n)$ is $J$-linked.
Then, (\ref{encaixa}) follows by applying Theorem \ref{mainBERevMC} to $(J; I_1,\dots, I_n)$.
 \end{proof}

To any primitive vector $w=(w_1,\dots, w_n)\in \Z_{\geq 1}^n$, we associate the following ideal of $\O_n$:
\begin{equation}\label{Jw}
J_w=\left\langle x_1^{\frac{w_1\cdots w_n}{w_1}},\dots, x_n^{\frac{w_1\cdots w_n}{w_n}} \right\rangle.
\end{equation}
Corollary \ref{corolBE} says, in particular, that, if $g:(\C^n,0)\to (\C^n,0)$ is a
semi-weighted homogeneous map with respect to $w$ (see \cite[p.\,793]{BENewton}),
then
$$
\LL_{J_w}(g)=\frac{\max\{d_w(g_1),\dots, d_w(g_n)\}}{w_1\cdots w_n},
$$
where $d_w(h)$ denotes the degree of $h$ with respect to $w$, for any $h\in\O_n$; that is,
$d_w(h)=\min\{\langle k,w\rangle: k\in\supp(h)\}$, where $\langle \,,\rangle$ denotes the standard scalar product.
As a consequence, if $f:(\C^n,0)\to (\C,0)$ is a semi-weighted homogeneous function with respect to $w$ (see \cite[p.\,793]{BENewton}) and
if we denote $\min\{w_1,\dots, w_n\}$ by $w_0$, then
$$
\LL_{J_w}(\nabla f)=\frac{d-w_0}{w_1\cdots w_n}.
$$

We recall that, by the main result of \cite{Brz}, if $f:(\C^n,0)\to (\C,0)$ is a semi-weighted homogeneous function such that $d_w(f)=d$, then
$\LL_0(\nabla f)=\frac{d-w_0}{w_0}$, provided that $d\geq 2w_i$, for all $i=1,\dots, n$.

\begin{thm}\label{LstarJnodeg}
Let $g=(g_1,\dots, g_n):(\C^n,0)\to (\C^n,0)$ be an analytic map germ such that
$g$ is $J$-non-degenerate and let $M=M_J$. Let $d_i=\nu_J(g_i)$, for all $i=1,\dots, n$, and let us suppose that $d_1\leq \cdots \leq d_n$.
Then
\begin{equation}\label{formulaLi}
\LL_J^{(i)}(g)\geq \frac{\max\{d_i, M\}}{M}
\end{equation}
for all $i=1,\dots, n-1$. If moreover $\langle g_1,\dots, g_n\rangle\subseteq \overline J$, then $M\leq d_1$ and equality holds in
\textnormal{(\ref{formulaLi})}, that is, $\LL_J^{(i)}(g)=\frac{d_i}{M}$, for all $i=1,\dots, n$.
\end{thm}

\begin{proof}
Let $I=\langle g_1,\dots, g_n\rangle$ and let $i\in\{1,\dots, n-1\}$.
By Proposition \ref{eideg}, we have that
$$
e_i(I, J)=\frac{d_1\cdots d_i}{M^i}e(J).
$$
Moreover, the number on the right of the previous equality can be also interpreted as
$$
\frac{d_1\cdots d_i}{M^i}e(J)=e(g_1,\dots, g_i, h_{i+1},\dots, h_n),
$$
where $(h_{i+1},\dots, h_n)$ is a sufficiently general element of $J\oplus \cdots \oplus J$ such that the
map $(g_1,\dots, g_i,h_{i+1},\dots, h_n)$ is $J$-non-degenerate.
Thus, by Proposition \ref{lower}, we obtain the inequality
\begin{equation}\label{fihj}
\frac{\max\{d_i, M\}}{M}=\LL_J(g_1,\dots g_i, h_{i+1},\dots, h_n)\leq \LL_J(\underbrace{I,\dots, I}_{i}, \underbrace{J,\dots, J}_{n-i})=\LL_J^{(i)}(I),
\end{equation}
where the first equality follows from Corollary \ref{corolBE}. Hence (\ref{formulaLi}) follows.

Let us suppose that $I\subseteq\overline J$. By Proposition \ref{decreix} we have that
\begin{equation}\label{Lisimple2}
\LL_J^{(i)}(I)=\inf\left\{\frac rs: r,s\in\Z_{\geq 1}, r\geq s,\,e_{i}(I^s+J^r, J)=e_{i}(I^s, J)\right\}.
\end{equation}
The inclusion $I\subseteq\overline J$ also implies that $M\leq d_1$. 
We claim that
\begin{equation}\label{Mdi}
e_{i}(I^M+J^{d_i}, J)=e_{i}(I^M, J).
\end{equation}
By (\ref{Lisimple2}), this would imply that $\LL_J^{(i)}(I)\leq \frac{d_i}{M}$
and hence $\LL_J^{(i)}(I)=\frac{d_i}{M}$, by (\ref{fihj}).

Let $k^{(1)},\dots, k^{(r)}\in \Z^n_{\geq 0}$ such that $x^{k^{(1)}},\dots,x^{k^{(r)}}$ is a minimal generating set of $J$.
Hence
$$
\overline{I^M+J^{d_i}}=\overline{\langle g_1^M,\dots, g_n^M, x^{d_ik^{(1)}},\dots,x^{d_ik^{(r)}}\rangle}
$$
(see for instance \cite[Proposition 8.15]{HS} or \cite[Corollary 1.40]{V}). By Proposition \ref{Prop2.9},
there exist generic $\C$-linear combinations $f_1,\dots, f_i$ of $\{g_1^M,\dots, g_n^M, x^{d_ik^{(1)}},\dots,x^{d_ik^{(r)}}\}$
and generic $\C$-linear combinations $f_{i+1},\dots, f_n$ of $\{x^{k^{(1)}},\dots,x^{k^{(r)}}\}$ such that
$\langle f_1,\dots, f_n\rangle$ has finite colength and
\begin{equation}\label{fis}
e_i(I^M+J^{d_i}, J)=e(f_1,\dots, f_i, f_{i+1},\dots, f_n).
\end{equation}
Since the ideals $I$ and $J$ have finite colength, the existence of such elements $f_1,\dots, f_n$ also follows by the theorem of
existence of joint reductions (see \cite[Theorem 1.4]{Rees2} or \cite[Theorem 1.6]{RS}).
Let $B$ denote the column matrix obtained as the
transpose of the matrix
$$
\left[\begin{matrix} g_1^M & \cdots & g_n^M & x^{d_ik^{(1)}} & \cdots & x^{d_ik^{(r)}}
\end{matrix}\right].
$$
Let $A$ denote a matrix of size $i\times (n+r)$ with entries in $\C$ such that
\begin{equation}\label{AB=F}
AB=
\left[\begin{array}{c}
f_1\\
\vdots \\
f_i
\end{array}
\right].
\end{equation}
Since the coefficients of $A$ are generic, we can assume that the submatrix $C$ formed the first
$i$ columns of $A$ is invertible. Therefore, by multiplying both sides of (\ref{AB=F}) by $C^{-1}$, we obtain that
\begin{equation}\label{sep}
\langle f_1,\dots, f_i\rangle=\langle f'_1,\dots, f'_i\rangle
\end{equation}
where
\begin{align*}
f'_1&=g_1^M+\sum_{j=i+1}^n\alpha_{1j}g_j^M+\sum_{\ell=1}^r\beta_{1\ell}x^{d_ik^{(\ell)}}\\
&\hspace{0.25cm}\vdots\\
f'_i&=g_i^M+\sum_{j=i+1}^n\alpha_{ij}g_j^M+\sum_{\ell=1}^r\beta_{i\ell}x^{d_ik^{(\ell)}}
\end{align*}
for some coefficients $\alpha_{sj}$, $\beta_{s\ell}\in \C$, $s=1,\dots, i$, $j=i+1,\dots, n$, $\ell=1,\dots, r$.
Since $d_1\leq \cdots \leq d_n$, we have that $\nu_J(f'_s)=d_sM$, for all $s=1,\dots, i$.
Then, from (\ref{fis}) and (\ref{sep}), we obtain the following:
\begin{align}
e_i(I^M+J^{d_i}, J)&=e(f_1,\dots, f_i, f_{i+1},\dots, f_n)=e(f'_1,\dots, f'_i, f_{i+1},\dots, f_n)\nonumber\\
&\geq\frac{(d_1M)\cdots (d_iM)M^{n-i}}{M^n}e(J)=d_1\cdots d_ie(J) \label{segona}
\end{align}
where the inequality of (\ref{segona}) is an application of (\ref{desref}).

By Proposition \ref{eideg}, we have
\begin{equation}\label{eiIM}
e_i(I^M, J)=M^i e_i(I, J)=d_1\cdots d_ie(J).
\end{equation}
Thus $e_i(I^M+J^{d_i}, J)\geq e_i(I^M, J)$.
Moreover, the inclusion $I^M\subseteq \overline{I^M+J^{d_i}}$ implies that
$e_i(I^M, J)\geq e_i(I^M+J^{d_i}, J)$.
Hence, we conclude that (\ref{Mdi}) is true and hence
the result follows.
\end{proof}


\begin{cor}\label{Calp} 
Let $g=(g_1,\dots, g_n):(\C^n,0)\to (\C^n,0)$ be an analytic map germ such that $g^{-1}(0)=\{0\}$. Let
$d_i=\nu_J(g_i)$, for all $i=1,\dots, n$, and let $M=M_J$. Let us suppose that
$M\leq d_1\leq \cdots \leq d_n$. Then, the following conditions are equivalent:
\begin{enumerate}
\item[(a)] $g:(\C^n,0)\to (\C^n,0)$ is $J$-non-degenerate.
\item[(b)] $\LL_J^{(i)}(g )=\frac{d_i}{M}$, for all $i=1,\dots, n$.
\end{enumerate}
In particular, if $g$ satisfies any of the above conditions, then $g$ is Hickel with respect to $J$.
\end{cor}

\begin{proof} Let $I=\langle g_1,\dots, g_n\rangle$. The implication (a) $\Rightarrow$ (b) follows as an immediate application of Corollary \ref{corolBE}, when $i=n$, and Theorem \ref{LstarJnodeg}, when $i<n$ (let us remark that the condition $M\leq d_1\leq \cdots \leq d_n$ implies
$\langle g_1,\dots, g_n\rangle\subseteq \overline J$).

Let us assume that (b) holds. Then, by (\ref{desref}) and (\ref{JHickel}), we have the following chain of inequalities:
$$
\frac{d_1\cdots d_n}{M^n}\leq \frac{e(g)}{e(J)}\leq \LL^{(1)}_J(I)\cdots \LL^{(n)}_J(I)=\frac{d_1\cdots d_n}{M^n}.
$$
Hence all inequalities become equalities, in particular (a) follows.

Let us suppose that (a) or (b) holds. Then
$$
\frac{e(g)}{e(J)}=\frac{d_1\dots d_n}{M^n}=\LL^{(1)}_J(I)\cdots \LL^{(n)}_J(I),
$$
where the first equality follows from (\ref{gJnodeg}) and the second follows from item (b). Thus, the ideal $I$ is Hickel with respect to $J$.
\end{proof}

\begin{rem}
\begin{enumerate}
\item Under the hypothesis of Corollary \ref{Calp}, if $g$ is Hickel with respect to $J$ then $g$ is not $J$-non-degenerate in general.
For instance, let us consider the map $g:(\C^2,0)\to (\C^2,0)$ given by $g(x,y)=(x^2+y^3, x^2-y^3)$. We observe that
$e(g)=6$, $\LL^{(1)}_0(g)=2$, $\LL^{(2)}_0(g)=3$, since $I(g)=\langle x^2, y^3\rangle$. Hence $g$ is Hickel with respect to $
\m_2$ but $g$ is not $\m_2$-non-degenerate.

\item Let $w=(w_1,\dots, w_n)\in\Z^n_{\geq 1}$.
It is known that if $g=(\C^n,0)\to (\C^n,0)$ is a weighted homogeneous map with respect to $w$ such that
$g^{-1}(0)=\{0\}$, then $\LL^*_0(g_1,\dots, g_n)$ is not always determined by $w$ and the vector of degrees $d_w(g)$ of $g$ with respect to $w$
(see for instance \cite[Example 4.3]{BF2}). However, as seen in Corollary \ref{Calp},
all numbers of the sequence $\LL^*_{J_w}(g_1,\dots, g_n)$ depend only on $w$ and $d_w(g)$, where $J_w$ is the ideal of $\O_n$ defined in (\ref{Jw}).
\end{enumerate}
\end{rem}


\section{The sequence $\LL_J^*(I)$ when $I$ and $J$ are monomial ideals}\label{Lstarmonomial}

Let $\sL\subseteq\{1,\dots, n\}$, $\sL\neq \emptyset$. Let $\vert \sL\vert$ denote the number of elements of $\sL$.
If $\K=\R$ or $\C$, then we denote by $\K^n_\sL$ the set of those $k\in\K^n$ such that $k_i=0$, for all $i\notin \sL$.
If $h\in\O_n$, $h\neq 0$, and $h=\sum_ka_kx^k$ is the Taylor expansion of $h$ around the origin, then we denote by $h^\sL$
the sum of those $a_kx^k$ such that $k\in \supp(h)\cap\R^n_\sL$.
We set $h^\sL=0$ if $\supp(h)\cap\R^n_\sL=\emptyset$.

Given an ideal $I$ of $\O_n$, we denote by $I^\sL$
the ideal of $\O_{\vert \sL\vert}$ generated by all elements $h^\sL$ such that $h\in I$.

Let us suppose that $\sL=\{i_1,\dots, i_r\}$, where $1\leq i_1<\cdots <i_r\leq n$.
Let $\pi_\sL:\C^r\to \C^n$ be the embedding defined by $\pi_\sL(x_{i_1},\dots, x_{i_r})=(y_1,\dots, y_n)$, where
$y_i=0$, if $i\notin\sL$, and $y_i=x_i$, if $i\in \sL$.
Let $\pi_\sL^*:\O_n\to \O_r$ be the morphism given by $\pi_\sL^*(h)=h\circ \pi_\sL$, for all $h\in \O_n$.
We observe that $h^\sL=\pi^*_\sL(h)$, for any $h\in \O_n$. This implies that
$I^\sL=\pi^*_\sL(I)$, for any ideal $I$ of $\O_n$.

\begin{lem}\label{restrict} Let $I$ and $J$ be ideals of $\O_n$ such that $\V(I)\subseteq \V(J)$. Let $\sL\subseteq\{1,\dots, n\}$, $\sL\neq\emptyset$. Then
$\LL_{J^\sL}(I^\sL)$ exists and
\begin{equation}\label{LIJL}
\LL_J(I)\geq \LL_{J^\sL}(I^\sL).
\end{equation}
\end{lem}

\begin{proof}
Since $\V(I)\subseteq \V(J)$, we have that $\LL_J(I)$ exists (see \cite{LT} or \cite{Tresonances}).
By (\ref{LJIintegral2}), let $p,q\in\Z_{\geq 1}$ such that $\overline{J^q}\subseteq \overline{I^p}$. Applying $\pi^*_\sL$ to both sides of this inclusion, we obtain that
$$
\pi^*_\sL(J^q)\subseteq\pi^*_\sL(\overline{J^q})\subseteq \pi^*_\sL(\overline{I^p})\subseteq\overline{\pi^*_\sL(I^p)},
$$
where the last inclusion follows from the persistence property of the integral closure of ideals (see \cite[p.\,2]{HS}).
Therefore, by relation (\ref{LJIintegral2}), inequality (\ref{LIJL}) follows.
\end{proof}

Alternatively, the previous result also arises as a direct consequence of the original formulation of \L ojasiewicz exponents by means of analytic inequalities
(see (\ref{Lojanalytic})).



Here we recall a definition from \cite{BiviaPLMS} (which in turn is very similar to \cite[Definition 3.1]{BiviaMRL}).

\begin{defn}\label{ndsequence}
Let $g_1,\dots, g_p\in \O_n$, where $p\leq n$. Let
$\Gamma_+$ denote the Minkowski sum $\Gamma_+(g_1)+\cdots
+\Gamma_+(g_p)$. Let $\Delta$ be a compact face of $\Gamma_+$.
The face $\Delta$ is univocally expressed as $\Delta=\Delta_1+\cdots+\Delta_p$, where $\Delta_i$
is a compact face of $\Gamma_+(g_i)$, for all $i=1,\dots, p$. We
say that the sequence $g_1,\dots, g_p$ satisfies the {\it
$(B_\Delta)$ condition} when
$$
\big\{x\in\C^n:
{(g_1)}_{\Delta_1}(x)=\cdots={(g_p)}_{\Delta_p}(x)=0\big\}
\subseteq\{x\in\C^n: x_1\cdots x_n=0\}.
$$
We say that $g_1,\dots, g_p$ is a {\it non-degenerate sequence} when the following conditions hold:
\begin{enumerate}
\item[(a)] the ring $\O_n/\langle g_1,\dots, g_p\rangle$ has dimension $n-p$;
\item[(b)] $g_1,\dots, g_p$ satisfy the $(B_\Delta)$ condition, for all
compact faces $\Delta$ of $\Gamma_+$ with $\dim(\Delta)\leq
p-1$.
\end{enumerate}
\end{defn}

Let $J$ be a monomial ideal of $\O_n$ of finite colength, $n\geq 2$, and
let $i\in\{0,1,\dots, n-1\}$. We denote by
$\G_i(J)$ the family of maps $(g_{i+1},\dots, g_n):(\C^n,0)\to (\C^{n-i},0)$ whose components constitute a non-degenerate sequence
and $\supp(g_j)=\v(\Gamma_+(J))$, for all $j=i+1,\dots, n$.

\begin{prop}\label{nodegmult}
Let $I_1,\dots, I_n$ be monomial ideals of $\O_n$ of finite colength. Let
$g_1,\dots, g_n\in\O_n$ such that $\Gamma_+(g_i)=\Gamma_+(I_i)$, for all $i=1,\dots, n$.
The following conditions are equivalent:
\begin{enumerate}
\item[(a)] $\langle g_1,\dots, g_n\rangle$ has finite colength and $e(g_1,\dots, g_n)=e(I_1,\dots, I_n)$;
\item[(b)] the sequence $g_1,\dots, g_n$ is non-degenerate.
\end{enumerate}
\end{prop}

\begin{proof}
It follows as an immediate application of \cite[Theorem 5.5]{BiviaPLMS} and \cite[Proposition 5.4]{BiviaPLMS}.
\end{proof}


\begin{rem} We point out that $\G_0(J)$ consists of those maps $g=(g_1,\dots, g_n):(\C^n,0)\to (\C^n,0)$ such that $g_1,\dots, g_n$ generate
a reduction of $J$ (see \cite[Proposition 3.6]{BiviaMRL}) and $\supp(g_j)=\v(\Gamma_+(J))$, for
all $j=1,\dots, n$. Moreover $\G_{n-1}(J)$ is formed by the functions of $J$ whose support is equal to $\v(\Gamma_+(J))$.
\end{rem}

With the aim of simplifying the notation, if $g=(g_1,\dots,g_p):(\C^n,0)\to (\C^p,0)$ is an analytic map
and $I$ is any ideal of $\O_n$, then we will denote the image of $I$ in the quotient ring $\O_n/\langle g_1,\dots, g_p\rangle$
by $I_g$.

\begin{lem}\label{aquocients}
Let $I_1,\dots, I_i$ be ideals of $\O_n$, $n\geq 2$, where $i\in\{1,\dots, n-1\}$. Let $g_{i+1},\dots, g_{n}\in\O_n$ such that the multiplicity
$\sigma(I_1,\dots, I_{i}, g_{i+1},\dots, g_n)$ is finite.
Let $g=(g_{i+1},\dots, g_n)$. Then $\sigma((I_1)_g,\dots, (I_i)_g)<\infty$ and
$$
\sigma(I_1,\dots, I_{i}, g_{i+1},\dots, g_n)=\sigma((I_1)_g,\dots, (I_i)_g).
$$
\end{lem}

\begin{proof}
Let $R=\O_n/\langle g_{i+1},\dots, g_n\rangle$ and let $p:\O_n\to R$ be the natural projection.
By Proposition \ref{Prop2.9}, there exists a sufficiently general element $(h_1,\dots, h_i)\in I_1\oplus \cdots \oplus I_i$
such that $\langle h_1,\dots, h_i, g_{i+1},\dots, g_n\rangle$ is an ideal of finite colength and
$\sigma(I_1,\dots, I_{i}, g_{i+1},\dots, g_n)=e(h_1,\dots,h_{i}, g_{i+1},\dots, g_n)$. Therefore
\begin{align}
\sigma(I_1,\dots, I_{i}, g_{i+1},\dots, g_n)&=e(h_1,\dots,h_{i}, g_{i+1},\dots, g_n)=\ell\left(\frac{\O_n}{\langle h_{1},\dots, h_i, g_{i+1}\dots, g_n\rangle}\right) \nonumber \\
&=\ell\left(\frac{R}{\langle p(h_1),\dots, p(h_i)\rangle}\right)=e\left( p(h_1),\dots, p(h_i)\right)   \nonumber\\
&\geq \sigma((I_1)_g,\dots, (I_i)_g). \label{extra}
\end{align}
Hence we have that $\sigma((I_1)_g,\dots, (I_i)_g)<\infty$ and $\dim \O_n/\langle g_{i+1},\dots, g_n\rangle=i$. Therefore, by Proposition 2.2,
the element $(h_1,\dots, h_i)$ can be taken in such a way that
equality holds in (\ref{extra}). Thus the result follows.
\end{proof}

Let us fix $i\in\{0,1,\dots, n-1\}$. Since $\v(\Gamma_+(J))$ is finite, the elements of $\G_{i}(J)$ are polynomial maps.
In the following result we identify each $g\in \G_{i}(J)$ with the family of coefficients of the components of $g$,
 so we can consider $\G_{i}(J)$ as a subset of a complex vector space of finite dimension. 


\begin{thm}\label{GiJexists}
Let $J$ be a monomial ideal of $\O_n$ of finite colength, $n\geq 2$, and
let $i\in\{0,\dots, n-1\}$. Then
$\G_{i}(J)$ contains a non-empty Zariski open set and
any $(g_{i+1},\dots, g_n)\in \G_{i}(J)$ verifies that
\begin{equation}\label{peratot}
\sigma(I_1,\dots, I_i, g_{i+1},\dots, g_n)=\sigma(I_1,\dots, I_i,J,\dots, J),
\end{equation}
for any family of monomial ideals $I_1,\dots, I_i$ of $\O_n$ such that $\sigma(I_1,\dots, I_i,J,\dots, J)<\infty$.
\end{thm}

\begin{proof}

Let us identify $\G_{i}(J)$ with a subset of $\C^{N(n-i)}$, where $N$ is the number of elements of $\v(\Gamma_+(J))$. Hence, $\G_i(J)$ contains the family of those maps
$(\C^n,0)\to (\C^{n-i},0)$ whose support coincides with $\v(\Gamma_+(J))$ and are Newton non-degenerate,
in the sense of \cite[Definition 3.8]{BiviaPLMS}.
Therefore, as a direct consequence of \cite[Lemma 6.11]{BiviaPLMS}, it follows that $\G_{i}(J)$ contains a non-empty Zariski open set.

Let us consider the case $i=0$ of relation (\ref{peratot}).
If $g=(g_1,\dots,g_n)\in \G_0(J)$, then $\langle g_1,\dots,g_n\rangle$ is Newton non-degenerate and $\Gamma_+(\langle g_1,\dots,g_n\rangle)=\Gamma_+(J)$.
In particular, $\overline{\langle g_1,\dots,g_n\rangle}=\overline J$, by \cite[Proposition 3.6]{BiviaMRL}, and thus $e(g_1,\dots,g_n)=e(J)$.

Let us suppose that $i>0$.
Let us fix any $g=(g_{i+1},\dots, g_n)\in \G_i(J)$.
Let $I_1,\dots, I_i$ be monomial ideals of $\O_n$ such that $\sigma(I_1,\dots, I_i, J,\dots, J)<\infty$.
Let us suppose first that $I_j$ has finite colength, for all $j=1, \dots, i$.
By Proposition \ref{Prop2.9}, there exists a sufficiently general element $(h_1,\dots, h_i)\in I_1\oplus \cdots \oplus I_i$
such that
\begin{equation}\label{sigmaIgs1}
\sigma(I_1,\dots, I_i, g_{i+1},\dots, g_n)=e(h_1,\dots, h_i,  g_{i+1},\dots, g_n).
\end{equation}

Since $g_{i+1},\dots, g_n$ is a non-degenerate sequence and $h_{i}$ is a generic $\C$-linear combination of a fixed generating system of
$I_i$, we can suppose, by \cite[Lemma 5.5]{BiviaPLMS}, that $h_i, g_{i+1},\dots, g_n$ is a non-degenerate sequence.
Inductively, we conclude that the elements $h_1,\dots, h_i$ can be chosen in such a way that (\ref{sigmaIgs1}) holds and
$h_1,\dots, h_i, g_{i+1},\dots, g_n$ is a non-degenerate sequence. The latter condition implies, by Proposition \ref{nodegmult}, that
\begin{equation}\label{sigmaIgs2}
e(h_1,\dots,h_i, g_{i+1},\dots, g_n)=e(I_1,\dots, I_i, J,\dots, J).
\end{equation}
By (\ref{sigmaIgs1}) and (\ref{sigmaIgs2}), it follows that $\sigma(I_1,\dots, I_i, g_{i+1},\dots, g_n)=e(I_1,\dots, I_i, J,\dots, J)$.

Let us suppose now that some of the ideals $I_1,\dots, I_i$ has not finite colength.
By the case discussed before, we have that
\begin{equation}\label{amberres}
\sigma(I_1+\m^r,\dots, I_i+\m^r, g_{i+1},\dots, g_n)=e(I_1+\m^r,\dots, I_i+\m^r, J,\dots, J)
\end{equation}
for all $r\in\Z_{\geq 1}$. By hypothesis, for any big enough $r$, the term on the right of (\ref{amberres}) is independent from $r$
and equal to $\sigma(I_1,\dots, I_i, J, \dots, J)$.
So the same happens with the multiplicity on the left of (\ref{amberres}).
Let us remark that, for any big enough $r\in\Z_{\geq 1}$, the following inequalities hold:
\begin{align*}
\sigma(I_1+\m^r,\dots, I_i+\m^r, g_{i+1},\dots, g_n)&\geq e(I_1+\m^r,\dots, I_i+\m^r, g_{i+1}+\m^r,\dots, g_n+\m^r)\\
&\geq e(I_1+\m^r,\dots, I_i+\m^r, J,\dots, J).
\end{align*}
Thus, by (\ref{amberres}), we have that $\sigma(I_1,\dots, I_i, g_{i+1},\dots, g_n)$ is finite and
$$
\sigma(I_1,\dots, I_i, g_{i+1},\dots, g_n)=\sigma(I_1,\dots, I_i, J,\dots, J).
$$
\end{proof}

\begin{prop}\label{LJImax}
Let $I$ and $J$ be monomial ideals of $\O_n$ of finite colength, $n\geq 2$.
Let $i\in\{1,\dots, n-1\}$ and let $\g\in \G_{i}(J)$. Then
\begin{equation}\label{suprem}
\LL_{J_\g}(I_\g)\leq \LL_J^{(i)}(I)
\end{equation}
and equality holds if $I\subseteq \overline J$.
\end{prop}

\begin{proof}
Let us fix an index $i\in\{1,\dots, n-1\}$ and let $\g=(g_{i+1},\dots, g_{n})\in \G_{i}(J)$. Then,
given two integers $r,s\geq 1$, we have the following inequalities:
\begin{align}
e_{i}(I^s, J^s)&=s^{n-i}e_{i}(I^s,J)    \nonumber\\
                    &=s^{n-i}\sigma(I^s,\dots, I^s, g_{i+1},\dots, g_n)         &&\textnormal{(by Theorem \ref{GiJexists})}   \nonumber\\
                    &=s^{n-i}e(I^s_\g)   &&\textnormal{(by Lemma \ref{aquocients})}\nonumber\\
                    &\geq s^{n-i}e\left((I^s+J^r)_\g\right) \label{ineqpelmig}\\
                    &=s^{n-i}\sigma\left(I^s+J^r,\dots,I^s+J^r,g_{i+1},\dots, g_n \right)   &&\textnormal{(by Lemma \ref{aquocients})}\nonumber \\
&= s^{n-i} e_{i}(I^s+J^r, J) &&\textnormal{(by Theorem \ref{GiJexists})}   \nonumber\\
&=e_{i}(I^s+J^r, J^s)\geq e_{i}(I^s+J^r, J^s+J^r).\nonumber
\end{align}

Hence, if $e_{i}(I^s, J^s)=e_{i}(I^s+J^r, J^s+J^r)$, then (\ref{ineqpelmig}) becomes an equality. That is, $e(I^s_\g)=e\left((I^s+J^r)_\g\right)$.
By (\ref{LJmillor}), this implies that $\LL_{J_\g}(I_\g)\leq \LL_J^{(i)}(I)$.

Let us suppose that $I\subseteq \overline J$. Hence $I_\g\subseteq (\overline{J})_\g\subseteq \overline{J_\g}$, for all $g\in \G_{i}(J)$, where
the second inclusion follows from the persistence of the integral closure under ring morphisms (see \cite[p.\,2]{HS}).
Thus $1\leq \LL_{J_\g}(I_\g)$ (see Remark \ref{LJgeq1}).

Let us suppose that $\LL_{J_\g}(I_\g)<\LL_J^{(i)}(I)$.
Then there exist some $r,s\in\Z_{\geq 1}$ such that $1\leq \LL_{J_\g}(I_\g) <\frac rs<\LL_J^{(i)}(I)$.
In particular, $r>s$ and the inequality $\frac rs<\LL_J^{(i)}(I)$ means that
\begin{equation}\label{bonic1}
e_{i}(I^s, J^s)>e_{i}(I^s+J^r, J^s+J^r).
\end{equation}
Since $r>s$, we obtain that
\begin{equation}\label{bonic3}
e_{i}(I^s+J^r, J^s+J^r)=e_{i}(I^s+J^r, J^s)=s^{n-i}e_{i}(I^s+J^r, J).
\end{equation}
Moreover
$e_{i}(I^s, J^s)=s^{n-i}e_{i}(I^s, J)$. Hence (\ref{bonic1}) and (\ref{bonic3}) imply that
\begin{equation}\label{bonic2}
e_{i}(I^s, J)>e_{i}(I^s+J^r, J).
\end{equation}

Since $g\in\G_i(J)$, by Lemma \ref{aquocients} and Theorem \ref{GiJexists}, the following equalities hold:
\begin{align}
e_{i}(I^s, J)&=e(I^s,\dots, I^s, g_{i+1},\dots, g_n)=e\left(I^s_\g\right)\label{un}\\
e_{i}(I^s+J^r, J)&=e(I^s+J^r,\dots,I^s+J^r,g_{i+1},\dots, g_n)=e\left({(I^s+J^r)_\g}\right).\label{dos}
\end{align}
The condition $\LL_{J_\g}(I_\g)<\frac rs$ means that $\overline{J^r_\g}\subseteq\overline{I^s_\g}$, which implies that
$e\left(I^s_\g\right)=e\left((I^s+J^r)_\g\right)$. Hence, by (\ref{un}) and (\ref{dos}),
we obtain that $e_{i}(I^s, J)=e_{i}(I^s+J^r, J)$, which contradicts (\ref{bonic2}).
Therefore we have that $\LL_{J_\g}(I_\g)=\LL_J^{(i)}(I)$.
\end{proof}


\begin{cor}\label{Ir}
Let $I$ and $J$ be monomial ideals of $\O_n$ of finite colength, $n\geq 2$. Let us suppose that $I\subseteq \overline J$.
Let $i\in\{1,\dots, n-1\}$. Then
\begin{equation}\label{rLI}
\LL_J^{(i)}(I^r)=r\LL_J^{(i)}(I)
\end{equation}
for any $r\in \Z_{\geq 1}$.

\end{cor}

\begin{proof}
Let us fix an $r\in\Z_{\geq 1}$. The relation $\LL_J(I^r)=r\LL_J(I)$ follows immediately from (\ref{LJIintegral2})
and holds for any pair of ideals $I$ and $J$ of finite colength of any local ring $R$. Let $i\in\{1,\dots, n-1\}$ and
let us fix an element $g\in \G_i(J)$. In particular, $\LL_{J_\g}(I^r_\g)=r\LL_{J_\g}(I_\g)$.
The condition $I\subseteq \overline J$ implies that $I_g^r\subseteq \overline J_g^r\subseteq \overline{J_g^r}$ and
consequently $\LL_{J_\g}(I^r_\g)=\LL_J^{(i)}(I^r)$, by Proposition \ref{LJImax}.
Hence (\ref{rLI}) follows.
\end{proof}


\begin{lem}\label{Lbaixa}
Let $n\geq 2$ and let us fix an index $i\in\{1,\dots, n-1\}$.
Let $f_1,\dots, f_i,g_{i+1},\dots, g_n$ be elements of $\O_n$ generating an ideal of finite colength in $\O_n$.
Let $\g=(g_{i+1},\dots, g_n)$.
If $J$ is any proper ideal of $\O_n$, then
\begin{equation}\label{ineqxula}
\LL_{J_\g}\big( \langle f_1,\dots, f_i \rangle_\g\big)\leq \LL_{J}(f_1,\dots, f_i,g_{i+1},\dots, g_n ).
\end{equation}
\end{lem}

\begin{proof}
Let $r,s\in\Z_{\geq 1}$. Then the following chain of inequalities holds:
\begin{align}
e(f_1^s,\dots, f_i^s,g_{i+1}^s,\dots, g_{n}^s)&=s^ne(f_1,\dots, f_i,g_{i+1},\dots, g_n) \nonumber\\
&=s^ne\big(\langle  f_1,\dots, f_i\rangle_\g\big)   &&\textnormal{(by Lemma \ref{aquocients})} \nonumber \\
&=s^{n-i}e\big(\langle  f_1^s,\dots, f_i^s\rangle_\g\big)    \nonumber    \\
&\geq s^{n-i}e\big(  (f_1^s+J^r)_\g,\dots, (f_i^s+J^r)_\g\big)                                              \label{obtenim}\\
&=s^{n-i}e\big( f_1^s+J^r,\dots, f_i^s+J^r, g_{i+1},\dots, g_n\big)  &&\textnormal{(by Lemma \ref{aquocients})}    \nonumber  \\
&=e\big( f_1^s+J^r,\dots, f_i^s+J^r, g_{i+1}^s,\dots, g_n^s\big)                                                  \nonumber\\
&\geq e\big( f_1^s+J^r,\dots, f_i^s+J^r, g_{i+1}^s+J^r,\dots, g_n^s+J^r\big).                                       \nonumber
\end{align}
If $e(f_1^s,\dots, f_i^s,g_{i+1}^s,\dots, g_{n}^s)=e\big( f_1^s+J^r,\dots, f_i^s+J^r, g_{i+1}^s+J^r,\dots, g_n^s+J^r\big)$,
then we obtain that (\ref{obtenim}) becomes an equality, which is to say that
$$
e\big(\langle f_1^s,\dots,f_i^s\rangle_\g\big)=e\big(\langle f_1^s,\dots,f_i^s\rangle_\g+J^r_\g\big).
$$
In particular, applying Definition \ref{mixedLoj} we obtain inequality (\ref{ineqxula}).
\end{proof}

\begin{thm}\label{essencial}
Let $I$ and $J$ be two monomial ideals of $\O_n$ of finite colength, $n\geq 2$.
Let $K_1,\dots, K_n$ be ideals of $\O_n$ contained in $\overline I$ such that $(K_1,\dots, K_n)$ is $J$-non-degenerate.
Let
$d_i=\nu_J(K_i)$, for all $i=1,\dots, n$, and let us suppose that $d_1\leq \cdots \leq d_n$. Then
$\LL_J(I)\leq \frac{d_n}{M_J}$.
If moreover $I\subseteq \overline J$, then
$$
\LL_J^{(i)}(I)\leq \frac{d_i}{M_J}
$$
for all $i\in\{1,\dots, n-1\}$.
\end{thm}

\begin{proof}
By Proposition \ref{Prop2.9}, we can consider an element
$(f_1,\dots, f_n)\in K_1\oplus \cdots \oplus K_n$ such that $e(f_1,\dots, f_n)=\sigma(K_1,\dots, K_n)$.
By Proposition \ref{lower}, we have that $\LL_J(f_1,\dots, f_n)\leq \LL_J(K_1,\dots, K_n)$.

Since $K_i\subseteq \overline I$, for all $i=1,\dots, n$, we have $\langle f_1,\dots, f_n\rangle \subseteq \overline I$. Thus $\LL_J(I)\leq \LL_J(f_1,\dots, f_n)$.
By Corollary \ref{corolBE} we deduce that $\LL_J(K_1,\dots, K_n)= \frac{d_n}{M_J}$.
Joining the above inequalities, we obtain the following:
$$
\LL_J(I)\leq \LL_J(f_1,\dots, f_n)\leq \LL_J(K_1,\dots, K_n)=\frac{d_n}{M_J}.
$$

Let us suppose that $I\subseteq\overline J$ and let us fix an index $i\in \{1,\dots, n-1\}$.
By Theorem \ref{GiJexists}, we can consider a map $\g=(g_{i+1}, \dots, g_n)\in \G_i(J)$.
The inclusion $I\subseteq\overline J$ implies
that $\LL_J^{(i)}(I)=\LL_{J_\g}(I_\g)$, by Proposition \ref{LJImax}.

By hypothesis, the $n$-tuple of ideals $(K_1,\dots, K_n)$ is $J$-non-degenerate. Therefore, by Corollary \ref{Jpasseja},
we have that $(K_1,\dots, K_i, J,\dots, J)$ is $J$-non-degenerate, where $J$ is repeated $n-i$ times. In particular
$\sigma(K_1,\dots, K_i, J,\dots, J)<\infty$. By Theorem \ref{GiJexists}, it follows that
$$\sigma(K_1,\dots, K_i, J,\dots, J)=\sigma(K_1,\dots, K_i, g_{i+1},\dots, g_n).$$
By virtue of Proposition \ref{Prop2.9}, we can consider a sufficiently general element $(f_1,\dots, f_i)\in K_1\oplus \cdots \oplus K_i$ such that
\begin{equation}\label{Kis}
\sigma(K_1,\dots, K_i, g_{i+1}, \dots, g_n)=e(f_1,\dots, f_i, g_{i+1}, \dots, g_n).
\end{equation}
Hence $e(f_1,\dots, f_i, g_{i+1}, \dots, g_n)=\sigma(K_1,\dots, K_i, J,\dots, J)$.

Therefore, by (\ref{Kis}) and the fact that $\nu_J(f_j)\geq d_j$, for all $j=1,\dots, i$, we have the following:
\begin{align}
\frac{d_1\cdots d_i M_J^{n-i}}{M_J^n}e(J)&=\sigma(K_1,\dots, K_i, J,\dots, J)=e(f_1,\dots, f_i, g_{i+1}, \dots, g_n) \nonumber\\
&\geq \frac{\nu_J(f_1)\cdots \nu_J(f_i)M_J^{n-i}}{M_J^n}e(J)\geq \frac{d_1\cdots d_i M^{n-i}}{M^n}e(J). \label{simplif}
\end{align}
where the first inequality of (\ref{simplif}) comes from (\ref{desref}). Hence the inequalities of (\ref{simplif})
become equalities, which means that $(f_1,\dots, f_i, g_{i+1}, \dots, g_n)$ is $J$-non-degenerate and $\nu_J(f_j)= d_j$, for all $j=1,\dots, i$.

Therefore, we deduce the following:
\begin{align}
\LL_J^{(i)}(I)&=\LL_{J_\g}(I_\g)\leq \LL_{J_\g}(\langle f_1,\dots, f_i\rangle_\g)  &&
\textnormal{(since $\langle f_1,\dots, f_i\rangle\subseteq \overline I$)}  \nonumber \\
&\leq \LL_{J}(f_1,\dots, f_i, g_{i+1},\dots, g_n) && \textnormal{(by Lemma \ref{Lbaixa})}  \nonumber\\
&=\frac{\max\{d_1,\dots, d_i, M_J\}}{M_J}.        && \textnormal{(by Corollary \ref{corolBE})}  \label{pelcorolBE}
\end{align}

The condition $f_j\in K_j\subseteq I\subseteq \overline J$ implies that $d_j\geq \nu_J(I)\geq \nu_J(J)=M_J$, for all $j=1,\dots, i$.
Then the member on the right side of
(\ref{pelcorolBE}) is equal to $\frac{d_i}{M_J}$ and the result follows.
\end{proof}

\section{Applications}\label{applications}

In this section we show an application of Theorem \ref{essencial} by means of a specific
family of ideals $(K_1,\dots, K_n)$ constructed in \cite{BiviaeIeJ}. In order to express this, we need to expose
some preliminary definitions in the next subsection. We will also show that the computation of the whole sequence
$\LL_J^*(I)$ is possible whenever $J$ is a diagonal ideal.

\subsection{A bound for the quotient of multiplicities of two monomial ideals and its relation with \L ojasiewicz exponents}

If $J$ is a monomial ideal of $\O_n$ of finite colength and $A$ is a closed subset of
$\R^n_{\geq 0}$, then we define $\nu_J(A)=\min\{\phi_J(k):k\in A\}$. We denote by $\Gamma(J)$ the Newton boundary of $\Gamma_+(J)$.

Let $h\in\O_n$. We will say that $h$ is {\it $J$-homogeneous} when $\nu_J(h)=\nu_J(x^k)$, for
any $k\in\supp(h)$. Given a map $g:(\C^n,0)\to (\C^p,0)$, we say that $g$ is {\it $J$-homogeneous} when
each component function of $g$ is $J$-homogeneous.

\begin{defn}\label{aiJI}\cite{BiviaeIeJ}
Let $I$ and $J$ be monomial ideals of $\O_n$ of finite colength. We define, for all $i\in\{1,\dots, n\}$, the following number:
$$
a_{i,J}(I)=\max\big\{ \nu_J\big( \Gamma_+(I) \cap C(\Delta) \big): \textnormal{ $\Delta$ is a compact face of $\Gamma_+(J)$ of dimension $n-i$}   \big\}.
$$
Therefore $a_{i,J}(I)\in\Q_{\geq 0}$, for all $i=1,\dots, n$.
It easily follows that $a_{1,J}(I)\leq\cdots \leq a_{n,J}(I)$.
\end{defn}

We will denote $a_{i,\m}(I)$ simply by $a_{i}(I)$, for all $i=1,\dots, n$.
Let us remark that the set of compact faces of $\Gamma_+(\m)$ is given by $\{\Gamma(\m)\cap \R^n_\sL: \sL\subseteq\{1,\dots, n\}, \sL\neq\emptyset\}$
and $\phi_\m(k)=\vert k\vert$, for all $k\in\R^n_{\geq 0}$. Therefore
\begin{equation}\label{aiI}
a_i(I)=\max\big\{ \ord(I^\sL): \sL\subseteq\{1,\dots, n\},\,\vert\sL\vert=n-i+1\big\}.
\end{equation}
Hence we recover the definition of the integers $a_i(I)$ given in \cite[p.\,197]{BiviaBAMS}.

Let $I$ be an ideal of $\O_n$ of finite colength and let $u\in\Z^n_{\geq 0}$, $u\neq 0$.
We denote by $k^I_u$ the point of intersection of $\Gamma(I)$ with the half-line $\{\lambda u:\lambda\in\R_{\geq 0}\}$.
Therefore, if $J$ is another monomial ideal of $\O_n$ of finite colength, we have
$$
a_{n,J}(I)=\max\big\{\phi_J(k_u^I):\,\, u\in \v(\Gamma_+(J))\big\}.
$$
We also observe that, under the conditions of Definition \ref{aiJI}, the maximum
that leads to the computation of $a_{i,J}(I)$ is attained at some point of $\v(\Gamma_+(I))\cup \{ k_u^I: u\in \v(\Gamma_+(J))\}$.

The point $k_u^I$ has rational coordinates, for all $u\in \Z^n_{\geq 0}$, $u\neq 0$. Hence,
we define
$$
c_J(I)=\min\left\{c\in\Z_{\geq 1}: c k_u^I\in\Z^n_{\geq 0},\, \textnormal{for all $u\in \v(\Gamma_+(J))$}\right\}.
$$

\begin{thm}\label{core}\cite{BiviaeIeJ}
Let $I$ and $J$ be monomial ideals of $\O_n$ of finite colength.
Let $c=c_J(I)$ and let $M=M_J$.
For any $i\in\{1,\dots, n\}$, let us consider the ideal
\begin{equation}\label{elsKi}
K_i=\big\langle x^k: k\in\supp\big(\overline{I^{cM}}\big),\, \phi_J(k)=a_{i, J}(\overline{I^{cM}})\big\rangle.
\end{equation}
Then $(K_1,\dots, K_n)$ is $J$-non-degenerate.
\end{thm}

The numbers $a_{i,J}(I)$ have the following property, proven in \cite[Theorem 4.12]{BiviaeIeJ}.

\begin{thm}\label{eIeJ}\cite{BiviaeIeJ}
Let $I,J\subseteq \O_n$ be monomial ideals of $\O_n$ of finite colength. Let $M=M_J$.
Then
\begin{equation}\label{qxulo}
\frac{e(I)}{e(J)}\leq \frac{a_{1,J}(I)\cdots a_{n,J}(I)}{M^n}
\end{equation}
and the following conditions are equivalent:
\begin{enumerate}
\item[(a)] equality holds in \textnormal{(\ref{qxulo})};
\item[(b)] there exists a $J$-homogeneous and $J$-non-degenerate polynomial map
$g=(g_1,\dots, g_n):(\C^n,0)\to (\C^n,0)$ and some $s\in\Z_{\geq 1}$ such that $\overline{I^s}=\overline{\langle g_1,\dots, g_n\rangle}$ and
$\nu_J(g_i)=s a_{i, J}(I)$, for all $i=1,\dots, n$.
\end{enumerate}
\end{thm}


As observed in \cite[Remark 5.6]{BiviaeIeJ}, when equality holds in (\ref{qxulo}), then the number $s$ appearing in item (b)
can be taken as $s=c_J(I)M_J$. In particular, we can take $s=1$ when $J=\m_n$.


\begin{cor}\label{eixmotor}
Let $I,J$ be monomial ideals of $\O_n$ of finite colength such that $I\subseteq \overline J$. Then
\begin{equation}\label{niceineq}
\LL_J^{(i)}(I)\leq \frac{a_{i, J}(I)}{M_J}
\end{equation}
for all $i\in\{1,\dots, n-1\}$.
\end{cor}

\begin{proof}
Let $M=M_J$, let $c=c_J(I)$ and let $\phi=\phi_J$.
Let us consider the ideals $K_1,\dots, K_n$ of $\O_n$ defined in (\ref{elsKi}).
By Theorem \ref{core}, we know that $(K_1,\dots, K_n)$ is $J$-non-degenerate. Moreover
$\nu_J(K_i)=a_{i, J}(I^{cM})$, for all $j=1,\dots, n$,
and $\nu_J(K_1)\leq \cdots \leq \nu_J(K_n)$.

Let us fix an index $i\in\{1,\dots, n-1\}$. By Theorem \ref{essencial}, we have that
\begin{equation}\label{IcM}
\LL_J^{(i)}(I^{cM})\leq \frac{a_{i, J}(I^{cM})}{M}=\frac{cM a_{i, J}(I)}{M}.
\end{equation}
Since we assume that $I\subseteq \overline J$, Corollary \ref{Ir} implies the equality $\LL_J^{(i)}(I^{cM})=cM\LL_J^{(i)}(I)$.
By joining this fact with (\ref{IcM}), relation (\ref{niceineq}) follows.
\end{proof}

In the following example we see that, in general, inequality (\ref{niceineq}) can be strict (we will see that this is not
the case when $J$ is diagonal).

\begin{ex}
Let us consider the ideals $I$ and $J$ of $\O_2$ given by $I=\langle x^5, y^5 \rangle$ and $J=\langle x^4, xy, y^4\rangle$.
We observe that $I\subseteq \overline J$, $M_J=4$ and $a_{1,J}(I)=\phi_J(\frac 52, \frac 52)=10$. So
$\frac{a_{1,J}(I)}{M_J}=\frac 52$.

A straightforward reproduction of the argument of the
proof of \cite[Corollary 3.4]{BiviaMZ2} consisting of replacing the powers of the maximal ideal by
the powers $J$ leads to the equality $\LL_J(I,J)=\LL_J(f,g)$, provided that $(f,g)$ is a sufficiently general element of $I\oplus J$
(see \cite[Theorem 3.6]{BE1}).
Let $H=\langle f,g\rangle$. Let $K_H$ denote the ideal of $\O_2$
generated by the monomials $x^{k_1}y^{k_2}$ which are integral over $H$, $k_1,k_2\in\Z_{\geq 0}$.
By applying \cite[Corollary 4.8]{BiviaCA2}, we observe that $K_H=\overline{\langle x^5, x^2y, xy^2, y^5\rangle}$.
The inclusion $K_H\subseteq \overline H$ implies that $\LL_J(H)\leq \LL_J(K_H)$.
It is easy to check that $a_{2,J} (K_H)=\phi_J(\frac 32,\frac 32)=6$. Then, we obtain the following inequalities:
$$
\LL_J^{(1)}(I)=\LL_J(I,J)=\LL_J(H)\leq \LL_J(K_H)=\frac{a_{2,J} (K_H)}{M_J}=\frac{3}{2}<\frac{5}{2}=\frac{a_{1, J}(I)}{M_J}.
$$
\end{ex}

In the study of examples, the computation of $\phi_J(k)$ for a given $k\in\Z^n_{\geq 0}$ can be done with the program {\it G\'ermenes} \cite{Montesinos}
developed by A.\,Montesinos-Amilibia.


\begin{cor}\label{corolfinal}
Let $I,J$ be monomial ideals of $\O_n$ of finite colength such that $I\subseteq \overline J$ and let $M=M_J$. Then
\begin{equation}\label{chain}
\frac{e(I)}{e(J)}\leq \LL_J^{(1)}(I)\cdots \LL_J^{(n)}(I)\leq \frac{a_{1,J}(I)\cdots a_{n,J}(I)}{M^n}
\end{equation}
and both inequalities turn into equalities if and only if there exists a polynomial map
$g=(g_1,\dots, g_n):(\C^n,0)\to (\C^n,0)$ and some $s\in\Z_{\geq 1}$ such that
$g$ is $J$-non-degenerate and $J$-homogeneous, $\overline{I^s}=\overline{\langle g_1,\dots, g_n\rangle}$ and $\nu_J(g_i)=s a_{i, J}(I)$,
for all $i=1,\dots, n$.
\end{cor}

\begin{proof}
The first inequality of (\ref{chain}) comes from (\ref{eIeJLJI}) and the second inequality of (\ref{chain})
is a direct application of Corollary \ref{eixmotor}. The characterization of when both inequalities of (\ref{chain}) become equalities
 follows from Theorem \ref{eIeJ}.
\end{proof}

\subsection{The sequence $\LL^*_J(I)$ when $J$ is diagonal}

Let us fix coordinates $(x_1,\dots, x_n)$ in $\C^n$. We say that an ideal $J\subseteq \O_n$
is {\it diagonal} when there exist positive integers $a_1,\dots, a_n$ such that
$\overline J=\overline{\langle x_1^{a_1},\dots, x_n^{a_n}\rangle}$.
In the next result we show some cases where (\ref{niceineq}) becomes an equality.

\begin{thm}\label{LmixtosJI}
Let $I,J$ be monomial ideals of $\O_n$ of finite colength. Then
\begin{equation}\label{casi=n}
\LL_J(I)=\frac{a_{n, J}(I)}{M_J}.
\end{equation}
If, moreover, $J$ is diagonal and $I\subseteq \overline J$, then
\begin{equation}\label{casdiagonal}
\LL_J^{(i)}(I)=\frac{a_{i, J}(I)}{M_J},
\end{equation}
for all $i=1,\dots, n$.
\end{thm}

\begin{proof}
Let us see first relation (\ref{casi=n}). Let $\phi=\phi_J$ and let $p,q\in\Z_{\geq 1}$.
As recalled in Section \ref{M&NF}, the integral closure of a monomial ideal of $\O_n$ is generated by
the monomials whose support belongs to the Newton polyhedron of the given ideal. Hence we have the following equivalences:
\begin{align*}
J^p\subseteq \overline{I^q}  &\Longleftrightarrow  \Gamma_+(J^p)\subseteq \Gamma_+(I^q)\\
                             &\Longleftrightarrow p\Gamma_+(J)\subseteq q\Gamma_+(I)\\
                             &\Longleftrightarrow \phi(pu)\geq \phi(qk_u^I),\,\textnormal{for all $u\in \v(\Gamma_+(J))$}\\
                             &\Longleftrightarrow p M\geq q\phi(k_u^I),\,\textnormal{for all $u\in \v(\Gamma_+(J))$}\\
                             &\Longleftrightarrow \frac{p}{q}\geq \frac{\phi(k_u^I)}{M},\,\textnormal{for all $u\in \v(\Gamma_+(J))$}.
\end{align*}
Therefore, by using (\ref{LJIintegral2}), we obtain that
$$
\LL_J(I)=\frac{\max\{\phi(k_u^I): u\in \v(\Gamma_+(J))\}}{M}=\frac{a_{n, J}(I)}{M}.
$$

Let us fix an index $i\in\{1,\dots, n-1\}$. Let us see that (\ref{casdiagonal}) holds provided that $J$ is diagonal and $I\subseteq \overline J$.

Let us suppose that $J$ is diagonal. Since $\LL_J^{(i)}(I)=\LL_{\overline J}^{(i)}(I)$ we can suppose that
there exist positive integers $a_1,\dots, a_n$ such that $J=\langle x_1^{a_1},\dots,x_n^{a_n}\rangle$.
By Proposition \ref{Prop2.9}, let us consider a sufficiently general element
$(h_1,\dots, h_i,g_{1},\dots, g_{n-i})\in I\oplus \cdots \oplus I\oplus J\oplus\cdots \oplus J$
such that
$$
e(h_1,\dots, h_i,g_{1},\dots, g_{n-i})=e(I,\dots, I, J,\dots, J).
$$

By Proposition \ref{lower}, we have that
\begin{equation}\label{intermig1}
\LL_J^{(i)}(I)=\LL_J(I,\dots, I,J, \dots, J)\geq\LL_J(h_1,\dots, h_i, g_{1},\dots, g_{n-i}).
\end{equation}
Let us denote by $A$ the ideal $\langle h_1,\dots, h_i, g_{1},\dots, g_{n-i}\rangle$
and let us fix a subset $\sL\subseteq\{1,\dots, n\}$ with $\vert \sL\vert=n-i+1$. In order to simplify the notation,
with no loss of generality,
we will suppose that $\sL=\{1,\dots, n-i+1\}$.
Hence, by Lemmas \ref{byarcs} and \ref{restrict}, we obtain that
\begin{equation}\label{intermig2}
\LL_J(A)\geq \LL_{J^\sL}(A^\sL)=\sup_{\varphi\in \Omega_{\sL}}\frac{\ord( \varphi^*(A^\sL))}{\ord( \varphi^*(J^\sL))}
\end{equation}
where $\Omega_{\sL}$ is the set of analytic arcs $\varphi:(\C,0)\to (\C^n_\sL,0)$.

Let us consider the map $\psi:(\C^n,0)\to (\C^n,0)$ given by $\psi(x_1,\dots, x_n)=(x_1^{a_1},\dots, x_n^{a_n})$.
Let us observe that each $g_j$ can be expressed as $g_j=f_j\circ \psi$, for all $j=1, \dots, n-i$, where $f_{1},\dots, f_{n-i}$
are generic linear forms of $\C[x_1,\dots, x_n]$.

The matrix of coefficients of the system of linear equations $f_{1}^\sL=\dots=f_{n-i}^\sL=0$ has size $(n-i)\times (n-i+1)$.
Applying the Gauss elimination process to this system, we conclude that there exist polynomials
$\overline g_{1},\dots, \overline g_{n-i}$ of the form
$$
\overline g_j(x_1,\dots, x_{n-i+1})= x_j^{a_j}-\gamma_j x_{n-i+1}^{a_{n-i+1}}
$$
for some $\gamma_j\in\C\smallsetminus\{0\}$, for all $j=1,\dots, n-i$, such
that
\begin{equation}\label{gesbarra}
\langle g_1^\sL,\dots, g_{n-i}^\sL\rangle=\langle \overline g_1,\dots, \overline g_{n-i}\rangle.
\end{equation}

Let $a=a_1\cdots a_n$. Let us consider the analytic arc $\varphi_0:(\C,0)\to (\C^n_\sL,0)$ given by
\begin{equation}\label{arg}
\varphi_0(t)=\big((\gamma_{1}t^a)^\frac{1}{a_1},
\dots,(\gamma_{n-i}t^{a})^\frac{1}{a_{n-i}}, t^{\frac{a}{a_{n-i+1}}}  \big)
\end{equation}
for all $t\in \C$. Let us write $\gamma_j=r_je^{\imag \theta_j}$, where $r_j\in\R_{>0}$, $\theta_j\in[0,2\pi[$, for all $j=1,\dots, n-i$, and in
(\ref{arg}) we consider the definition $\gamma_{j}^{1/a_j}=r_{j}^{1/a_j}e^{\imag \theta_{j}/a_j}$, for all $i=1,\dots, n-i$.
We observe that
\begin{equation}\label{vanishes}
(\overline g_j\circ\varphi_0)(t)=\left((\gamma_{j}t^a)^\frac{1}{a_j}\right)^{a_j}-\gamma_j \left(t^{\frac{a}{a_{n-i+1}}}\right)^{a_{n-i+1}}=
\gamma_j t^a-\gamma_j t^a=0
\end{equation}
for all $t\in \C$ and all $j=1,\dots, n-i$.

The compact face of dimension $n-1$ of $\Gamma_+(J)$ is supported by the vector $w=(\frac{a}{a_1},\dots,\frac{a}{a_n})$.
This vector has integer coordinates but is not primitive in general. Let $v$ denote the smallest vector of the form
$v=\lambda w$, $\lambda>0$, such that $v$ is primitive. Hence $v=\frac{1}{w_0}w$, where $w_0$
denotes the greatest common divisor of the components of $w$.
Moreover $\phi_J(k)=\langle v,k\rangle$, for all $k\in\R^n_{\geq 0}$
and $M_J=\frac{a}{w_0}$.

For any $f\in\O_n$, $f\neq 0$, we define $\nu'_J(f)=\min\{\langle w,k\rangle: k\in\supp(f)\}$.
We also set $\nu'_J(0)=+\infty$. Then $\nu'_J(f)=w_0 \nu_J(f)$, for all $f\in\O_n$.

Clearly we have that
\begin{equation}\label{ordreJL}
\ord\left( \varphi_0^*(J^\sL)\right)=\min\left\{ \ord(x_j^{a_j}\circ \varphi_0): j=1,\dots, n-i+1   \right\}=a.
\end{equation}
Moreover
$$
A^\sL=\left\langle  h_1^\sL,\dots, h_i^\sL, g_{1}^\sL,\dots, g_{n-i}^\sL\right\rangle=
\left\langle  h_1^\sL,\dots, h_i^\sL\right\rangle+ \left\langle \overline g_{1},\dots, \overline g_{n-i}\right\rangle
$$
where the last equality comes from (\ref{gesbarra}).
Then
$$
\ord\left( \varphi_0^*(A^\sL)\right)=\min\left\{ \ord(h_1^\sL\circ \varphi_0),\dots, \ord(h_i^\sL\circ \varphi_0),
\ord(\overline g_{1}\circ \varphi_0),\dots, \ord(\overline g_{n-i}\circ \varphi_0)\right\}.
$$
By (\ref{vanishes}), we know that $\overline g_j\circ \varphi_0=0$, for all $j=1,\dots, n-i$.
Since the coefficients of the forms $f_1,\dots, f_{n-i}$ are chosen generically, we can assume that the numbers $\gamma_1,\dots, \gamma_{n-i}$
verify that all the arcs $h_1^\sL\circ \varphi_0,\dots, h_i^\sL\circ \varphi_0$ are non-zero.
Hence, we conclude that
\begin{equation}\label{ordreAL}
\ord\left( \varphi_0^*(A^\sL)\right)=
\min\left\{ \ord(h_1^\sL\circ \varphi_0),\dots, \ord(h_i^\sL\circ \varphi_0)\right\}=\min\left\{ \nu'_J(h_1^\sL),\dots, \nu'_J(h_i^\sL)\right\}.
\end{equation}

Thus we finally obtain, by (\ref{intermig1}), (\ref{intermig2}), (\ref{ordreJL}) and (\ref{ordreAL}) that
\begin{align*}
\LL_J^{(i)}(I)=\LL_J(A)\geq \LL_{J^\sL}(A^\sL)&\geq \frac{\ord( \varphi_0^*(A^\sL))}{\ord( \varphi_0^*(J^\sL))}\\
&=\frac{\min\left\{ \nu'_J(h_1^\sL),\dots, \nu'_J(h_i^\sL)\right\}}{a}
=\frac{\nu_J(I^\sL)w_0}{M_Jw_0}=\frac{\nu_J(I^\sL)}{M_J}.
\end{align*}

That is, we have proved that
$$
\LL_J^{(i)}(I)\geq \frac{\nu_J(I^\sL)}{M_J}
$$
for all $\sL\subseteq \{1,\dots, n\}$ such that $\vert \sL\vert =n-i+1$. This means that
$$
\LL_J^{(i)}(I)\geq \frac{1}{M_J}\max\left\{\nu_J(I^\sL): \sL\subseteq\{1,\dots, n\}, \vert \sL\vert=n-i+1\right\}=\frac{a_{i, J}(I)}{M_J}.
$$

We have already proved in Corollary \ref{eixmotor} that the inequality $\LL_J^{(i)}(I)\leq \frac{a_{i, J}(I)}{M_J}$ holds in general.
Therefore equality (\ref{casdiagonal}) follows.
\end{proof}

\begin{rem}
\begin{enumerate}
\item[(a)] Given any diagonal ideal $J$ of $\O_n$ and $i\in\{0,1,\dots,  n-1\}$, the compact faces of $\Gamma_+(J)$ of dimension $i$ are given
by $\{\Gamma(J)\cap \R^n_\sL: \sL\subseteq\{1,\dots, n\}, \vert \sL\vert=i+1\}$.
In particular, analogous to (\ref{aiI}), if $I$ denotes any monomial ideal of $\O_n$ of finite colength, we have:
$$
a_{i, J}(I)=\max\big\{\nu_J(I^\sL\O_n):  \sL\subseteq\{1,\dots, n\big\}, \vert \sL\vert=n-i+1\}.
$$

\item[(b)] Let $I$ be a monomial ideal of $\O_n$ of finite colength.
As a direct application of (\ref{aiI}) and Theorem \ref{LmixtosJI} in the case $J=\m_n$, we obtain that
\begin{equation}\label{eqLojorder}
\LL_0^{(i)}(I)=\max\big\{\ord(I^{\sL}): \sL\subseteq\{1,\dots, n\},\,\vert \sL\vert=n-i+1\big\},
\end{equation}
for all $i=1,\dots, n$. The above relation was already proven in \cite[Corollary 4.2]{BF1} by means of a completely different argument
based on toric modifications.
\end{enumerate}
\end{rem}

Let $J$ be a diagonal ideal of $\O_n$ given by $J=\overline{\langle x_1^{a_1},\dots, x_n^{a_n}\rangle}$, where
$a_1,\dots, a_n\in\Z^n_{\geq 1}$. Let $w_J=(\frac{a_1\cdots a_n}{a_1},\dots, \frac{a_1\cdots a_n}{a_n})$ and let
$v_J=\frac{1}{w_0}w_J$, where $w_0$ denotes the greatest common divisor of the components of $w_J$.
Then the filtrating map $\phi_J:\R^n_{\geq 0}\to \R_{\geq 0}$ is given by $\phi_J(k)=\langle v_J,k\rangle$, for all $k\in\R^n_{\geq 0}$. Therefore $M_J=\frac{a_1\dots a_n}{w_0}$.

\begin{cor}\label{corolfinalqh}
Under the conditions and notation of the above paragraph,
let $I$ be another monomial ideal of $\O_n$ of finite colength such that $I\subseteq \overline J$. Then
\begin{equation}\label{chain2}
\frac{e(I)}{e(J)}\leq \LL_J^{(1)}(I)\cdots \LL_J^{(n)}(I)=\frac{a_{1,J}(I)\cdots a_{n,J}(I)}{\big(\frac{a_1\cdots a_n}{w_0}\big)^n}
\end{equation}
and the following conditions are equivalent:
\begin{enumerate}
\item[(a)] equality holds in \textnormal{(\ref{chain2})};
\item[(b)] there exists some $s\geq 1$ such that $\overline{I^{s}}=\overline{\langle g_1,\dots, g_n\rangle}$,
where $g=(g_1,\dots, g_n):(\C^n,0)\to (\C^n,0)$ is a weighted homogeneous map
with respect to $v_J$ such that $d_{v_J}(g_i)=s a_{i, J}(I)$, for all $i=1,\dots, n$.
\end{enumerate}
\end{cor}

\begin{proof}
By Theorem \ref{LmixtosJI} we know that $\LL_J^{(i)}(I)=\frac{a_{i, J}(I)}{M_J}$, for all $i=1,\dots, n$. Hence the result follows as an immediate
consequence of Theorem \ref{eIeJ}.
\end{proof}

If $J$ is a diagonal ideal, then it follows immediately from the definition of $c_J(I)$ that $c_J(I)=1$, for any monomial ideal $I$ of finite colength of $\O_n$. Therefore, as remarked after Theorem \ref{eIeJ}, when equality holds in (\ref{chain2}), then the integer $s$ appearing in item (b) of
the previous corollary can be taken as $s=c_J(I)M_J=\frac{a_1\cdots a_n}{w_0}$.

\begin{ex} Let us consider the ideals $J$ and $I$ of $\O_3$ given by
$J=\langle x^a, y^b, z^c\rangle$ and $I=\langle x^d, y^d, z^d, x^ey^ez^e\rangle$,
where $a,b,c,d,e$ are positive integers such that $1\leq a\leq b\leq c\leq d$ and $abc\leq e(bc+ac+ab)\leq abd$.
These conditions imply that $\nu_J(I)=\phi_J(e,e,e)$ and $I\subseteq \overline J$. By Theorem \ref{LmixtosJI}, the sequence $\LL_J^*(I)$ is given by
$$
\LL_J^*(I)=\left(\LL_J^{(3)}(I),\LL_J^{(2)}(I),\LL_J^{(1)}(I)\right)=\left(\frac da, \frac db, \frac{e(bc+ac+ab)}{abc}\right).
$$
By Corollary \ref{corolfinalqh}, we know that
$\frac{e(I)}{e(J)}\leq \LL_J^{(1)}(I)\LL_J^{(2)}(I)\LL_J^{(3)}(I)$. We observe that $e(I)=3d^2e$ and $e(J)=abc$. Therefore
$$
\frac{e(I)}{e(J)}\leq \LL_J^{(1)}(I)\LL_J^{(2)}(I)\LL_J^{(3)}(I)\Longleftrightarrow
\frac{3d^2e}{abc}\leq \frac{d^2e(bc+ac+ab)}{a^2b^2c} \Longleftrightarrow 3\leq \frac ca+\frac cb+1,
$$
which is the case, since we assume that $a\leq b\leq c$. Moreover equality holds if and only if $a=b=c$.
\end{ex}

\begin{ex}
Let us consider the diagonal ideals of $\O_n$ given by $J=\overline{\langle x^{a_1}_1,\dots, x^{a_n}_n\rangle}$ and
$I=\overline{\langle x_1^{b_1},\dots, x_n^{b_n}\rangle}$, where $a_i,b_i\in \Z_{\geq 1}$, $b_i\geq a_i$, for all $i=1,\dots, n$ (so that $I\subseteq \overline J$).
Let us consider a permutation $\sigma$ of $\{1,\dots, n\}$ such that
$\frac{b_{\sigma(1)}}{a_{\sigma(1)}}\leq \cdots \leq \frac{b_{\sigma(n)}}{a_{\sigma(n)}}$.
Then, as a consequence of Theorem \ref{LmixtosJI} we have that
$\LL_J^{(i)}(I)=\frac{b_{\sigma(i)}}{a_{\sigma(i)}}$, for all $i=1,\dots, n$.
\end{ex}

\begin{rem}\label{concl}
If $I$ and $J$ are monomial ideals of $\O_n$ of finite colength, then $\LL_J(I)$ can be also computed by means of the
Newton filtration induced by $\Gamma_+(I)$. That is, by \cite[Proposition 5.3]{BF1}, we have that $\LL_J(I)=\frac{M_I}{\nu_I(J)}$.
Joining this fact with (\ref{casi=n}), if $I\subseteq \overline J$, then we obtain the following relation:
$$
\LL_J(I)=\frac{M_I}{\nu_I(J)}=\frac{a_{n,J}(I)}{M_J}.
$$
\end{rem}

Let $I$ and $J$ be any pair of ideals of $\O_n$ of finite colength.
In accordance with inequality (\ref{eIeJLJI}) and the results of Hickel \cite[p.\,643]{Hickel}, there arises the problem of studying if the
condition
\begin{equation}\label{eqeIeJ}
\frac{e(I)}{e(J)}=\LL_J^{(1)}(I)\cdots \LL_J^{(n)}(I)
\end{equation}
determines some structure for the integral closure of $I$.
We conjecture that equality (\ref{eqeIeJ}) holds if and only if there exists some $s\geq 1$ such that
$\overline{I^s}=\overline{\langle g_1,\dots, g_n\rangle}$, where $(g_1,\dots, g_n)$ is $J$-non-degenerate (see \cite[Definition 4.7]{BiviaeIeJ}), that is,
there exists some $d\geq 1$ and some $a_1,\dots, a_n\in\Z_{\geq 1}$ such that
$\overline{\langle g_1^{a_1},\dots, g_n^{a_n}\rangle}=\overline{J^d}$.
By Corollary \ref{corolfinal}, we know that this is true
if $I$ is monomial and $J$ is diagonal.

Let $J$ be a diagonal ideal of $\O_n$.
In the following example we show that if $I$ is not a monomial ideal of $\O_n$
 and is Hickel with respect to $J$, that is, equality holds in (\ref{chain2}), then
we can not expect the same characterization appearing in Corollary \ref{corolfinalqh}. More precisely, in the context of Corollary \ref{corolfinalqh},
the condition of $J$-homogeneity of the map $g$ is too strong if $I$ is not monomial.

\begin{ex}\label{exempledeSaitama}  
Let us consider the ideal $I$ of $\O_2$ given by $I=\langle x^4+y^2, x^5\rangle$. We observe
that $e(I)=10=\LL_0^{(1)}(I)\LL_0^{(2)}(I)$.
The set of vertices of $\Gamma_+(I)$ is $\{(4, 0),(0, 2)\}$. If there exist two homogeneous polynomials $g_1,g_2\in\C[x,y]$ such that
$\overline{I^s}=\overline{\langle g_1,g_2\rangle}$, for some $s\geq 1$,
then $\deg(g_1)=4s$ and $\deg(g_2)=2s$, or vice versa. The Newton boundary $\Gamma(I)$
is formed by the segment joining the points $(4, 0)$ and $(0, 2)$. Hence it follows that $g_1$ or $g_2$ is a monomial.
Moreover, the condition $\overline{I^s}=\overline{\langle g_1,g_2\rangle}$ implies that $\langle g_1,g_2\rangle$
has finite colength. Thus $\langle g_1,g_2\rangle$ is Newton non-degenerate (see Remark \ref{NNDideals}).
In particular, the ideal $I^s$ is Newton non-degenerate. So
$I$ must be Newton non-degenerate too, which is not the case. Then the initial assumption is not true, that is,
the equivalence of Corollary \ref{corolfinalqh} does not hold in general if $I$ is not a monomial ideal and $J$ is equal to the maximal ideal.
\end{ex}




\end{document}